\documentclass[a4paper, 11pt]{article} 
\usepackage[utf8]{inputenc}
\usepackage{lmodern}
\usepackage{geometry}
\usepackage{amssymb, amsfonts, amsthm, amsmath, mathtools}
\usepackage[english]{babel}
\usepackage[colorlinks]{hyperref}

\numberwithin{equation}{section}

\newcommand{\dd}{\mathrm{d}}

\renewcommand{\bar}[1]{\overline{#1}}
\newcommand{\egaldistr}{\ {\overset{(d)}{=}}\ }
\renewcommand{\tilde}[1]{\widetilde{#1}}

\renewcommand{\rho}{\varrho}
\renewcommand{\epsilon}{\varepsilon}

\usepackage{xcolor}

\usepackage{caption}
\usepackage{tikz}
\usetikzlibrary{decorations}
\usetikzlibrary{decorations.pathmorphing}
\usetikzlibrary{decorations.pathreplacing}
\usetikzlibrary{decorations.shapes}
\theoremstyle{plain}
\newtheorem{theorem}{Theorem}[section]
\newtheorem{lem}{Lemma}[section]

\newtheorem{corollary}{Corollary}[section]

\newtheorem{remarque}{Remark}[section]
\newtheorem{proposition}{Proposition}[section]

\theoremstyle{definition}

\providecommand{\MR}{\relax\ifhmode\unskip\space\fi MR }

\usepackage{ulem}
\usepackage{textcomp}

\title{Computing the Yaglom limit of Markov chains with a single exit state using their excursion measure}
\author{Elie~Cerf\footnote{Université Sorbonne Paris Nord, LAGA, CNRS, UMR 7539, Laboratoire d'excellence Inflamex, F-93430 Villetaneuse, France.
		cerf@math.univ-paris13.fr} }
\date{\today}

\begin{document}
	
\maketitle
	
\begin{abstract}
	We prove in this article the existence of the Yaglom limit for Markov chains on discrete state spaces in the setting where the absorbing state is accessible from a single non-absorbing state. We use a representation of the trajectories of this process by its excursion away from death, that allows us to link the Yaglom limit with the large deviations behaviour of the inverse of its local time at the exit state, and to compute its minimal quasi-stationary distribution with its excursion measure.
\end{abstract}

\section{Introduction}

We consider a continuous-time Markov chain $(X_t)_{t\geq0}$ evolving on a discrete state space $E\cup\{\delta\}$ where $\delta$ is an absorbing state. Without loss of generality we can fix $E=\mathbb{N}=\{0,1,2,\dots\}$. We assume that the process is irreducible on $\mathbb{N}$ and that it is almost surely eventually absorbed, meaning that if we denote by $T_\delta:=\inf\{t\geq0, X_t=\delta\}$ the time of absorption, then for any initial distribution $\mu$
\begin{equation}\label{eqn:cond-killedas}\tag{$\mathcal{B}$}
	\mathbb{P}_\mu(T_\delta<\infty)=1.
\end{equation}
The asymptotic behaviour of $(X_t)_{t\geq0}$ being trivial, we are interested in its \textit{quasi-limiting} behaviour which is linked to the process conditionally on \textit{surviving} i.e on the law of $X_t$ conditionally on $\{t<T_\delta\}$. More precisely, we want to understand when the following limits are well defined 
\begin{equation}\label{eqn:defYaglomlim}
	\underset{t\to\infty}{\lim}\mathbb{P}_j(X_t=i|t<T_\delta), \qquad \forall i,j\in\mathbb{N},
\end{equation}
where $\mathbb{P}_j$ stands for the distribution of $(X_t)_{t\geq0}$ starting from state $j\in\mathbb{N}$. These objects are called Yaglom limits of the process $(X_t)_{t\geq0}$ and have been proven to be closely related to its quasi-stationary behavior. In fact, when the limit $\eqref{eqn:defYaglomlim}$ exists and does not depend on the starting point $j\in\mathbb{N}$ then it defines a quasi-stationary distribution (QSD) $\nu^*$ of the process \cite{VereJonesYaglom}, meaning that $\nu^*$ is a probability distribution on $\mathbb{N}$ such that 
\begin{equation}\label{eqn:defqsd}
	\mathbb{P}_{\nu^*}(X_t=i|t<T_\delta) = \nu^*(i), \qquad \forall i\in\mathbb{N}.
\end{equation}
By this definition, we remark that if a probability measure $\nu$ is a QSD for the process $(X_t)_{t\geq0}$, then the time of absorption $T_\delta$ is exponentially distributed under $\mathbb{P}_\nu$ \cite[Theorem 2.2. p.~19]{ColletMartinez_book}. So a necessary condition for the existence of a QSD is the existence of an exponential moment for the time of absorption $T_\delta$ \cite{good_1968} namely
\begin{equation}\label{eqn:conditionexpokilling}
	\gamma^*:=\sup\{\gamma>0, \exists i\in E : \mathbb{E}_i[e^{\gamma T_\delta}]<\infty\}>0.
\end{equation}

In the case of a finite state space, Darroch and Seneta \cite{DarrochSenetafiniteqsd} proved that this condition of exponential killing is always verified and that the minimal QSD exists and is the unique QSD of the Markov chain. However, the matter of the existence of QSDs and Yaglom limits is not trivial for an infinite state space. Therefore, several works have been focused on finding assumptions on the process $(X_t)_{t\geq0}$ under which the condition $\eqref{eqn:conditionexpokilling}$ is also sufficient for the existence of QSDs. This is the case for birth-and-death processes, as it has been shown by Good \cite{good_1968} and van Doorn \cite{vanDoornBDP91}. Then, for a general result, Ferrari et al. \cite{Ferrarirenewalapp} showed through the study of renewal processes linked to the absorption time of the Markov process that under the hypothesis $\eqref{eqn:cond-killedas}$, if $T_\delta$ also verifies \[ \underset{i\to\infty}{\lim} \mathbb{P}_i(T_\delta<t)=0,\] for all $t>0$, then the condition $\eqref{eqn:conditionexpokilling}$ is necessary and sufficient. Very recently, Yamato \cite{yamato2023existence} extended this result to downward skip-free Markov chains by looking at the excursions of the process away from any state $i\in\mathbb{N}$. But the renewal approach for the study Yaglom limits have also proved useful in the case continuous state spaces after the founding work of Mandl \cite{Mandl1961} on diffusions, and Yamato \cite{YamatoRenewaldiffusion} even showed how to recover information on non-minimal QSDs of general diffusions through the dynamical approach.

Our work is also inscribed in this framework. More precisely, we make the following assumption on the process $(X_t)_{t\geq0}$ 
\begin{equation}\label{hyp:excursions}\tag{$\mathcal{A}$}
	\text{\textit{The death-state $\delta$ is accessible from state $0$ only.}}
\end{equation}
This assumption can be written using the jump rate matrix $(q_{i,j})_{i,j\in\mathbb{N}\cup\{\delta\}}$ of the chain as: \[q_{i,\delta} = \lambda \mathbf{1}_{\{i=0\}}, \qquad \forall i\in\mathbb{N}, \] with $\lambda>0$ the jump rate from $0$ to $\delta$. Note that the choice of $0$ as unique \textit{exit-state} is arbitrary, the main hypothesis is that there is a unique entrance to $\delta$. Under this hypothesis we show the existence of the unique minimal QSD.
\begin{theorem}\label{thm:mainintro}
	Suppose that $(\mathcal{A})$ is verified and let $\lambda>0$ be the jump rate from $0$ to $\delta$. Then, there exists $\lambda_c>0$ such that for $\lambda<\lambda_c$ the condition $\eqref{eqn:conditionexpokilling}$ is sufficient for the existence and independence from the initial state of the Yaglom limit of $(X_t)_{t\geq0}$. In particular, its minimal quasi-stationary distribution $\nu^*$ exists and is given by \[ \nu^*(i) = \lim_{t\to\infty}\mathbb{P}_0(X_t=i|t<T_\delta), \qquad \forall i\in\mathbb{N}. \] Moreover, it verifies $\nu^*(0) = \frac{\gamma^*}{\lambda}$, and the $(\nu^*(i))_{i\geq1}$ can be computed using the excursion law away from state $0$.
\end{theorem}

Our method is based on the construction of $(X_t)_{t\geq0}$ as a killed version of a process $(\bar{X}_t)_{t\geq0}$ on $\mathbb{N}$ with transition given by $Q|_{\mathbb{N}\times\mathbb{N}}$ with $\lambda=0$. When $\eqref{hyp:excursions}$ is verified we can represent $T_\delta$ as $\inf \{t>0, L_t\geq T^{(\lambda)}\}$ where $T^{(\lambda)}$ is an exponential random variable independent of $\bar{X}$ with parameter $\lambda$ and \[L_t := \int_{0}^{t} \mathbf{1}_{\{\bar{X}_s = 0\}} \mathrm{d}s, \qquad \forall t\geq0 \] is the local time at $0$ of $(\bar{X}_t)_{t\geq0}$. The process $(L_t)_{t\geq0}$ is a continuous and increasing process such that for any time $t\geq0$, $L_t$ represents the time that the chain has spent in state $0$ before $t$. In other words, \[ (X_s)_{0\leq s\leq T_\delta} \egaldistr (\bar{X}_s)_{0\leq s \leq \sigma_{T^{(\lambda)}}}, \] where $\sigma_s:=\inf\{t>0, L_t\geq s\}$ is the right-continuous inverse of $(L_t)_{t\geq0}$.

It is well-known \cite[Theorem 8, p.~114]{BertoinLevyBook} that the inverse local time $(\sigma_s)_{s\geq0}$ is a subordinator. We denote by $\psi$ its Laplace exponent and we set \[ \theta_+ = \sup \{\theta>0, \psi(\theta)<\infty\}, \] and prove in Section \ref{sec:appqsd} that 
\begin{equation*}
	\lambda_c = \psi(\theta_+) = \lim_{\theta\to\theta_+} \psi(\theta) \in (0,\infty].
\end{equation*}
We point out that $\lambda<\lambda_c$ is always verified when $\psi$ does not jump to $\infty$.

The rest of the article is organised as follows. In the next section, we study of large deviations for subordinators. Then, we formally construct the excursions of the process $(X_t)_{t\geq0}$ and prove Theorem \ref{thm:mainintro}. We conclude with an example of explicit computation in a case of a finite state space. Finally we use these estimates to compute the limit distributions of the overshoot and undershoot of a killed subordinator at the passage of a given level.

\section{Large Deviations for a killed subordinator}

The goal of this section is to obtain two estimates on the large deviations of a killed subordinator. We start by introducing subordinators and some of their properties. Then, we recall some estimates for random walks before extending them to subordinators and finally to killed subordinators. 

\subsection{Subordinators}

A subordinator is a Lévy process starting from $0$ and with non-decreasing trajectory. We give in this section a selected set of properties verified by subordinators and refer the reader to \cite{BertoinStFlour} and \cite[Chapter 5]{KyprianouLevy} for a complete introduction of subordinators and Lévy processes in general. 

We start by recalling a fundamental property on the Laplace transform of subordinators. As Lévy processes, subordinators admits a Lévy-Khintchine decomposition for their characteristic function \cite[Lem. 2.14. p.~57]{KyprianouLevy} from which we can derive a characterization of the Laplace exponent of subordinators.

\begin{theorem}\label{prop:laplacesubor}
	Let $(Y_t)_{t\geq0}$ be a subordinator. Its Laplace transform verifies for all $t\geq0$ and $\theta\in\mathbb{R}$
	\begin{equation}\label{eqn:laplacetransformsubor}
		\mathbb{E}[e^{\theta Y_t}] = e^{t \psi(\theta)},
	\end{equation}
	where the function $\psi : \mathbb{R} \to \mathbb{R}\cup\{\infty\}$ is called the Laplace exponent of $(Y_t)_{t\geq0}$. \\
	Moreover, there exist a unique real number $d\geq0$ and a unique measure $\Pi$ on $(0,\infty)$ verifying $\int_{0}^{\infty} (1\wedge x) \Pi(\mathrm{d}x)<\infty$ such that $\psi$ verifies for all $\theta\in\mathbb{R}$
	\begin{equation}\label{eqn:LKforsubor}
		\psi(\theta) = d \theta + \int_{0}^{\infty} (e^{\theta x}-1) \Pi(\mathrm{d}x).
	\end{equation}
\end{theorem}

We conclude this section by introducing compound Poisson processes with drift as a simple class of subordinators. A process $(Y_t)_{t\geq0}$ is a compound Poisson process if there exist $(N_t)_{t\geq0}$ a Poisson process with intensity $a$ and an independent family $(\xi_n)_{n\in\mathbb{N}}$ of i.i.d random variables with distribution $\eta$ on $(0,\infty)$ such that 
\begin{equation}\label{eqn:compoundPoisson}
	Y_t = \sum_{i=1}^{N_t} \xi_i \qquad \forall t\geq0.
\end{equation}
It is easy to verify that such a process is a subordinator. Moreover, we can easily compute its Laplace transform using the ones of the Poisson process $(N_t)_{t\geq0}$ and of $\eta$. 

\begin{lem}\label{lem:psipourpoisson}
	The Laplace exponent $\psi$ of $(Y_t)_{t\geq0}$ defined in $\eqref{eqn:compoundPoisson}$ is given by 
	\begin{equation}\label{eqn:laplacecompoundpoisson}
		\psi(\theta) = a \int_{0}^{\infty} (e^{\theta x} - 1) \eta(\mathrm{d}x) \qquad \forall \theta\in\mathbb{R}.
	\end{equation}
	In particular, the Lévy measure of $(Y_t)_{t\geq0}$ is given by $\Pi(\mathrm{d}x) = a \eta(\mathrm{d}x)$ and has finite mass, and the drift is null.
\end{lem}

\begin{proof}
	For any $\theta>0$, we have \[ \psi(\theta) = \log \mathbb{E}[e^{\theta \sum_{i=1}^{N_1} \xi_i }], \] but $(N_t)_{t\geq0}$ is a Poisson process, so $N_1$ follows a Poisson distribution $\mathcal{P}(a)$. Therefore, writing $\Psi_{\eta}$ the Laplace transform of the $(\xi_n)_{n\in\mathbb{N^*}}$ we get \[ \mathbb{E}[e^{\theta \sum_{i=1}^{N_1} \xi_i}] = e^{a(\int_{0}^{\infty}e^{\theta x} \eta(\mathrm{d}x) -1)}, \] which gives \[ \psi(\theta) = a\int_{0}^{\infty}(e^{\theta x}-1) \eta(\mathrm{d}x). \] \qedhere
\end{proof}

As a direct corollary, we get that a subordinator $(Y_t)_{t\geq0}$ is a compound Poisson process if and only if it has zero drift and its Lévy measure has finite mass. We define a compound Poisson process with drift as a process $(Y_t)_{t\geq0}$ on $(0,\infty)$ which can be written \[ Y_t = dt + \sum_{i=1}^{N_t} \xi_i, \] i.e. a subordinator whose Lévy measure has finite mass.

\subsection{Large deviations for random walks}

Let $(S_n)_{n\in\mathbb{N}}$ be a random walk on $\mathbb{R_+}$ i.e. a discrete time process with i.i.d non-negative increments. We denote by $\psi$ the log-Laplace transform of $S_1$ meaning $$\psi(\theta)=\log\mathbb{E}[e^{\theta S_1}]\in\mathbb{R}\cup\{+\infty\}, \mbox{ for all } \theta\in\mathbb{R}.$$ We set $$\theta_+:=\sup\{\theta>0 : \mathbb{E}[e^{\theta S_1}]<\infty\}\geq0.$$
Then, since $S_1$ is a.s positive , the function $\psi$ is finite on $(-\infty,\theta_+)$. \\
\indent The study of the large deviations of $(S_n)_{n\in\mathbb{N}}$ is strongly linked to the Legendre transform of the function $\psi$ defined by $$\psi^*(a):=\underset{\theta\geq0}{\sup}(a\theta-\psi(\theta)).$$ In fact, the function $\psi^*$ is the \textit{rate function} for the large deviations, in other words:
\begin{equation*}
	\frac{1}{n}\log(\mathbb{P}(S_n>na))\underset{n\rightarrow\infty}{\longrightarrow}-\psi^*(a), \mbox{ for } a>0,
\end{equation*} 
see \cite{DemboZeitouni} or \cite{IscoeNeyNummelin} for example. We recall the following classic properties of $\psi$ and $psi^*$, that can be found in Borovkov's book \cite[Chapter 9]{Borovkov}. 
\begin{proposition}\label{prop:rappelsurlestransf}
	Suppose $\theta_+>0$, then the following properties are verified.
	\begin{enumerate}
		\item The Laplace transform $\psi$ is continuous on $[0,\theta_+]$ (even if $\psi(\theta_+)=+\infty)$. Moreover, it is increasing, convex and analytic on the interval $[0,\theta_+)$.
		\item For any $a\in\mathbb{R_+}$ there exists a unique $\theta_a\in[0,\theta_+]$ such that $\psi^*(a)=a\theta_a-\psi(\theta_a)$.
		\item Writing $a_+:=\underset{\theta\rightarrow\theta_+}{\lim}\psi'(\theta)$, we have
		\begin{equation*}
			\theta_a=\begin{cases}
				0 & \text{ if } a<\mathbb{E}[S_1]\\ \theta_+ & \text{ if } x>a_+\\ {\psi'}^{-1}(a) & \text{ otherwise.} 
			\end{cases}
		\end{equation*}
		It implies that $\theta_\cdot$ is analytic and increasing on $(\mathbb{E}[S_1],a_+)$.
		\item Let $x_+:=\sup\{ x\in\mathbb{R}: \mathbb{P}(S_1\leq x)<1\}$. The function $\psi^*$ is finite, continuous and strictly convex in $[\mathbb{E}[S_1],x_+]$, and it is analytic in $(\mathbb{E}[S_1],x_+)$. Additionally, ${\psi^*}^{'}(a)=\theta_a.$
	\end{enumerate}
\end{proposition}

Furthermore, Borovkov gives in the same book \cite[Theorem 9.3.2, p.~258]{Borovkov} the following Integro-Local theorem on large deviations for discrete time random walks, using notations from Proposition \ref{prop:rappelsurlestransf}.

\begin{theorem}\label{thm_int_loc}
	Using assumptions and notation from Proposition \ref{prop:rappelsurlestransf}. For any $a\in[\mathbb{E}[S_1], a_+ [$, we have
	\begin{equation}\label{pgd-ma}
		\mathbb{P}(S_n\geq an) \underset{n\rightarrow\infty}{\sim} e^{-n\psi^*(a)} \frac{1}{\theta_a\sqrt{2\pi\psi^{''}(\theta_a)n}}.
	\end{equation}
	Besides, the equivalence $\eqref{pgd-ma}$ is uniform in $[\mathbb{E}[S_1]+ \frac{A(n)}{\sqrt{n}},a_1[$ for any $A(n)\to\infty$ such that $\frac{A(n)}{\sqrt{n}}\to0$, and $a_1\in(\mathbb{E}[S_1],a_+)$.
\end{theorem}

To apply Theorem \ref{thm_int_loc} in continuous-time settings, in particular to subordinators, we use Croft-Kingman's lemma \cite{Kingman63}.

\begin{lem}\label{lem:kingman}
	Let $f:\mathbb{R}_+\to \mathbb{R}$ be a right-continuous function. Suppose that for any $t>0$ we have $\underset{n\in\mathbb{N}\rightarrow\infty}{\lim}f(nt) = 0,$ then $\underset{t\rightarrow\infty}{\lim}f(t)=0.$
\end{lem}

Using Lemma \ref{lem:kingman} with Theorem \ref{thm_int_loc} we obtain the following estimate for subordinator. 

\begin{proposition}\label{prop:continuous-time-integrolocal}
	Let $(Y_t)_{t\geq0}$ be a subordinator. Keeping previous notation for Laplace transforms, assume $\theta_+>0$. Then, for any $a\in[\mathbb{E}[Y_1],a_+)$, we have
	\begin{equation}\label{eqn:continuous-integrolocal}
		 \mathbb{P}(Y_t\geq ta)\underset{t\rightarrow\infty}		{\sim} e^{-t\psi^*(a)}\frac{1}{\theta_a\sqrt{2\pi\psi^{''}(\theta_a)t}}.
	\end{equation}
	Besides, the equivalence $\eqref{eqn:continuous-integrolocal}$ is uniform in $a\in[a_1,a_2]$ with $\mathbb{E}[Y_1]<a_1<a_2<a_+$.
\end{proposition}

\begin{proof}
	Let us prove $\eqref{eqn:continuous-integrolocal}$ by applying Croft-Kingman's Lemma to the function $$g:t\mapsto\left|\mathbb{P}(Y_t\geq ta)\sqrt{t}e^{\psi^*(a)t}-\frac{1}{\theta_a\sqrt{2\pi\psi^{''}(\theta_a)}}\right|$$ for any fixed $a\in[\mathbb{E}[Y_1],a_+)$. It is clear that $g$ is continuous on $\mathbb{R_+}$. Fix $t>0$, to prove the convergence of $g$ along its discrete-time skeleton $(g_{nt})_{n\in\mathbb{N}}$, let us denote for $n\in\mathbb{N}$, $\tilde{Y}_n:=Y_{nt}$. The process $(Y_t)_{t\geq0}$ being a Lévy process, it follows that the process $(\tilde{Y}_n)_{n\in\mathbb{N}}$ is a discrete-time random walk. Applying Theorem $\ref{thm_int_loc}$, as $ta\in[\mathbb{E}[\tilde{Y}_1],\tilde{a}_+)$ we have
	\begin{equation}\label{cvgence_skelet}
		\mathbb{P}(\tilde{Y}_n>nta)\underset{n\rightarrow\infty}{\sim} e^{-n\tilde{\psi}^*(ta)}\frac{1}{\tilde{\theta}_{ta}\sqrt{2\pi\tilde{\psi}^{''}(\tilde{\theta}_{ta})}},
	\end{equation}
	where $\tilde{\psi}(\theta)=\log\mathbb{E}[e^{\theta\tilde{Y}_1}]$, for any $\theta\in\mathbb{R}$. \\
	Additionally, by Proposition \ref{prop:laplacesubor} we have for $\theta\in\mathbb{R}$: $$\tilde{\psi}(\theta) =\log\mathbb{E}[e^{\theta Y_{nt}}] = \log\mathbb{E}[e^{\theta Y_1}]^{t}= t\psi(\theta),$$
	so we get for $a\in[\mathbb{E}[Y_1],a_+)$,  $$\tilde{\psi}^*(a) = \underset{\theta}{\sup}(\theta a - \tilde{\psi}(\theta)) = \underset{\theta}{\sup} (t\theta\frac{a}{t}-t\psi(\theta)) = t\psi^*(\frac{a}{t}).$$
	Besides, in $\eqref{cvgence_skelet}$, $\tilde{\theta}_{a}$ is such that: $$t\psi^*(a) = \tilde{\psi}^*(ta) = ta\tilde{\theta}_{ta} - \tilde{\psi}(\tilde{\theta}_{ta}) = ta\tilde{\theta}_{ta} - t\psi(\tilde{\theta}_{ta}).$$ Those computations being done for any $t$ and $a$, it follows that for $t>0$ and $a\in[\mathbb{E}[Y_1],a_+)$: $$ \tilde{\theta}_{ta} = \theta_{a};$$ where $\theta_{a}$ is such that, $$\psi^*(a)=a\theta_{a} - \psi(\theta_a).$$ Finally, equation $\eqref{cvgence_skelet}$ becomes, for $t>0$ and $a$ such that $ta\in[\mathbb{E}[\tilde{Y}_1,\tilde{a}_+)$ : $$\mathbb{P}(Y_{nt}>nta) \underset{n\rightarrow\infty}{\sim} e^{-nt\psi^*(a)}\frac{1}{\theta_{a}\sqrt{2\pi\psi^{''}(\theta_{a})}}.$$ But, we know from Proposition \ref{prop:rappelsurlestransf} $(3)$ that $$\tilde{a}_+=\underset{\theta\rightarrow\tilde{\theta}_+}{\lim}\tilde{\psi}'(\theta)= t\underset{\theta\rightarrow\tilde{\theta}_+}{\lim}\psi'(\theta)=ta_+,$$ since $\tilde{\theta}_+=\theta_+$. Moreover, for $t>0$ $$\mathbb{E}[Y_t]=\frac{\mathrm{d}}{\mathrm{d}\theta}(e^{t\psi(\theta)})\bigg\rvert_{\theta=0} =t\psi^{'}(0)=t\mathbb{E}[Y_1],$$ such that $$\mathbb{E}[\tilde{Y}_1]=t\mathbb{E}[Y_1].$$ It proves that for $a\in[\mathbb{E}[Y_1],a_+)$ $$g(nt)\underset{n\rightarrow\infty}{\longrightarrow}0, \mbox{for all } t>0.$$ Therefore, Croft-Kingman's Lemma ensures the convergence: $$g(t)\underset{t\rightarrow\infty}{\longrightarrow}0,$$ which gives $\eqref{eqn:continuous-integrolocal}$. \\ 
	Now let $\mathbb{E}[Y_1]<a_1<a_2<a_+$. To prove that the convergence is uniform for $a\in[a_1,a_2]$ we show that the function $$\tilde{g}:t\mapsto \underset{a\in[a_1,a_2]}{\sup}\left|\mathbb{P}(Y_t>ta)\sqrt{t}e^{\psi^*(a)t}-\frac{1}{\theta_a\sqrt{2\pi\psi^{''}(\theta_a)}}\right|,$$ also converges to $0$ when $t$ tends to $+\infty$. Theorem \ref{thm_int_loc} already gives us that the convergence $\eqref{cvgence_skelet}$ is uniform for $a\in[a_1,a_2]$, so following previous computations we have the convergence of $\tilde{g}$ along its skeleton : $$\tilde{g}(nt)\underset{n\rightarrow\infty}{\longrightarrow}0, \mbox{for all } t>0.$$ But we still need to prove that $\tilde{g}$ is continuous. Let $\epsilon>0$ and $t,s>0$, we have
	\begin{equation*}
		\begin{split}
			|\tilde{g}(t)-\tilde{g}(s)| &\leq \underset{a\in[a_1,a_2]}{\sup} |\mathbb{P}(Y_t>ta)\sqrt{t}e^{\psi^*(a)t} - \mathbb{P}(Y_s>as)\sqrt{s}e^{\psi^*(a)s}|, \\
			& = \underset{a\in[a_1,a_2]}{\sup} |h(t,a)-h(s,a)|,
		\end{split}
	\end{equation*} with $h:(r,a)\mapsto\mathbb{P}(Y_r>ra)\sqrt{r}e^{\psi^*(a)r}$. The function $h$ being continuous on $\mathbb{R}^2_+$, it is uniformly continuous on $[a_1,a_2]$, so we can define its modulus of continuity on this segment $\omega_{[a_1,a_2]}^h : [0,\infty] \to [0,\infty]$. We get
	\begin{equation*}
		\begin{split}
			|\tilde{g}(t)-\tilde{g}(s)| &\leq \underset{a\in[a_1,a_2]}{\sup} \omega_{[a_1,a_2]}^h(d_{\mathbb{R}^2_+}((t,a),(s,a))), \\
			&= \omega_{[a_1,a_2]}^h(|t-s|).
		\end{split}
	\end{equation*}
	But $\omega_{[a_1,a_2]}^h$ is continuous and $\omega_{[a_1,a_2]}^h(0)=0$ so the last inequality gives us the continuity of $\tilde{g}$ on $\mathbb{R}^*_+$. Finally, applying Lemma \ref{lem:kingman} to $\tilde{g}$ completes the proof.
\end{proof}

We now extend these results to the case of killed subordinators.

\subsection{Large deviations for a killed subordinator}\label{sect:LDkilledsub}

In our work, we are interested in subordinators that can be absorbed after a random time, we call them killed subordinators. Let $(Y_t)_{t\geq0}$ be a subordinator and let $T$ be an independent exponentially distributed random variable with parameter $\lambda$. We define the corresponding killed subordinator $(Z_t)_{t\geq0}$ by 

\begin{equation*}
	Z_t = \left\{
	\begin{array}{ll}
		Y_t & \mbox{if } t<T\\
		-\infty & \mbox{else}
	\end{array}
	\right. \qquad \forall s\geq0.
\end{equation*}

We remark from this definition that the Laplace exponent of a killed subordinator also admits a Lévy-Khintchine decomposition of the form $\eqref{eqn:LKforsubor}$: if we denote by $\psi$ the Laplace exponent of $(Y_t)_{t\geq0}$ and by $\psi^{(\lambda)}$ the one of $(Z_t)_{t\geq0}$ we have

\begin{equation}\label{eqn:lienlaplacekilled}
	\psi^{(\lambda)}(\theta) = \psi(\theta) - \lambda \qquad \forall \theta\in\mathbb{R}.
\end{equation}

We get directly that $\theta^{(\lambda)}_+ = \theta_+$ and we deduce the following result as a simple corollary from Proposition $\ref{prop:continuous-time-integrolocal}$. 

\begin{proposition}\label{prop:integrolocalkilled}
	Assume $\theta_+>0$. Then, for $a\in[\mathbb{E}[Y_1],a_+)$, we have the following asymptotic behaviour: 
	\begin{equation}\label{eqn:killed-continuous-integrolocal}
		\mathbb{P}(Z_t\geq ta)\underset{t\rightarrow\infty}	{\sim} e^{-t{\psi^{(\lambda)}}^*(a)}\frac{1}{\theta^{(\lambda)}_a\sqrt{2\pi{\psi^{(\lambda)}}^{''}(\theta^{\lambda}_a)t}}.
	\end{equation}
	Besides, the limit $\eqref{eqn:killed-continuous-integrolocal}$ is uniform for $a\in[a_1,a_2]$ with $\mathbb{E}[Y_1]<a_1<a_2<a_+$.
\end{proposition}

\begin{proof}
	Let $t>0$ and $a\in[\mathbb{E}[Y_1],a_+)$, we have $$\mathbb{P}(Z_t\geq ta)=\mathbb{P}(Y_t\geq ta,T>t)=\mathbb{P}(Y_t\geq ta)e^{-\lambda t}.$$ So we apply Proposition \ref{prop:continuous-time-integrolocal} and get 
	$$\mathbb{P}(Z_t\geq ta)\underset{t\rightarrow\infty}{\sim}e^{-t(\psi^*(a)+\lambda)}\frac{1}{\theta_a\sqrt{2\pi\psi^{''}(\theta_a)t}}.$$
	But, equation $\eqref{eqn:lienlaplacekilled}$ also implies $${\psi^{(\lambda)}}^*(a)=\underset{\theta>0}{\sup}(\theta a-\psi^{(\lambda)}(\theta)) = \psi^*(a)+\lambda,$$ then, from Proposition $\ref{prop:rappelsurlestransf}$ $(2)$ we also get that $\theta^{(\lambda)}_a=\theta_a$.
\end{proof}

\begin{remarque}
	The derivatives of $\psi$ and $\psi^{(\lambda)}$ being equal we omit the exponent $(\lambda)$ in these derivatives in the rest of the article. The same goes for the derivatives of ${\psi^{(\lambda)}}^*$, $a^{(\lambda)}_+$ and $\theta^{(\lambda)}_a$ for $a\in(0,a_+)$
\end{remarque}

\indent The first step to understand the asymptotic behaviour of $(Z_s)_{s\geq0}$ conditionally on surviving is to compute the limiting probability for $(Z_s)_{s\geq0}$ to die above level $t$, meaning that $Z_{T_-}>t$. One can easily see that $\mathbb{P}(Z_{T_-}>t)$goes to $0$ as $t$ tends to infinity, but since $\{Z_{T_-}>t\}=\{Y_{T_-}>t\}$ we can obtain more information on the speed of convergence. Set

\begin{equation}\label{eqn:deftheta*Z}
	\theta^*:=\sup\{\theta>0 : \mathbb{E}[e^{\theta Z_{T_-}}]<\infty\}
\end{equation}

The following lemma exhibits two different behaviour for $\theta^*$.

\begin{lem}\label{prop:explosionLaplaceZT}
	Assume $\theta_+>0$. If $\psi(\theta_+):=\lim_{\theta\to\theta_+}\psi(\theta)>\lambda$, then $\theta^*$ is the unique solution of 
	\begin{equation}\label{eqn:comportementtheta*}
		\psi(\theta) = \lambda, \qquad \theta\in(0,\theta_+). 
	\end{equation}
	Else, $\theta^*=\theta_+$.
\end{lem}

\begin{proof}
	Let $\theta>0$. Integrating over the value of $T$ we have \[ \mathbb{E}[e^{\theta Z_{T_-}}] = \mathbb{E}[e^{\theta Y_T}] = \lambda \int_{0}^{\infty} e^{-t(\lambda - \psi(\theta))} \mathrm{d}t. \] The function $t\mapsto e^{-t(\lambda - \psi(\theta))}$ is integrable if and only if $\psi(\theta)<\lambda$. Therefore, by monotony and continuity of $\psi$, $\theta^*$ is given by the smallest value of $\theta$ such that $\psi(\theta)\geq\lambda$. So there is only two possibilities: either $\psi(\theta_+):=\lim_{\theta\to\theta_+}\psi(\theta)>\lambda$ which is equivalent to $\theta^*$ being the unique solution to $\eqref{eqn:comportementtheta*}$, or $\psi$ jumps to $+\infty$ before reaching $\lambda$ and $\theta^*=\theta_+$.
\end{proof}

Going further we can use Proposition \ref{prop:continuous-time-integrolocal} to compute the exact speed of convergence of the probability to die above level $t$ as $t\to\infty$. More precisely, we have the following theorem. 

\begin{theorem}\label{thm_GD}
	Assume $\theta_+>0$ and $\psi(\theta_+):=\lim_{\theta\to\theta_+}\psi(\theta)>\lambda$. Then, as $t\rightarrow\infty$, the probability for $(Z_s)_{s\geq0}$ to die above level $t$ satisfies
	\begin{equation}\label{equi_ZT}
		\mathbb{P}(Z_{T_-}>t) \underset{t\rightarrow\infty}{\sim} \frac{\lambda u^*}{\theta^*}e^{-t\theta^*},
	\end{equation}
	where $u^*$ is the unique element of $(\frac{1}{a_+},\frac{1}{\mathbb{E}[Y_1]})$ such that $\theta^*=\theta_{\frac{1}{u^*}}$.
\end{theorem}

\begin{proof}
	Our goal is to compute an asymptotic equivalent for the integral
	\begin{equation*}
		\begin{split} 
			I(t):= \mathbb{P}(Z_{T_-}>t) = \mathbb{P}(Y_{T_-}>t)&= t\int_{0}^{\infty} \lambda e^{-\lambda ut} \mathbb{P}(Y_{ut}>t)\mathrm{d}u, \\ 
			&= t\int_{0}^{\infty} \lambda e^{-\lambda ut} \mathbb{P}(Y_{ut}>ut\frac{1}{u})\mathrm{d}u. 
		\end{split} 
	\end{equation*}
	We use Laplace's method and show that as $t\to\infty$ the mass of the integral above concentrates in an interval centred at $u^*$ in which we can use the uniform equivalent of Proposition \ref{prop:continuous-time-integrolocal}. \\
	First, let us introduce the function $f:u\mapsto -\lambda u - u\psi^*(\frac{1}{u})$. We prove that $f$ is maximal at $u^*$. Proposition \ref{prop:rappelsurlestransf} gives the following expression for $f$ on $\mathbb{R_+}$: 
	\begin{equation*}
		f(u)=\begin{cases}
			- \theta_+ + u(\psi(\theta_+)-\lambda) & \text{ if } u\in[0,\frac{1}{a_+}]\\ -\theta_{\frac{1}{u}} + u(\psi(\theta_{\frac{1}{u}}) - \lambda) & \text{ if } u\in(\frac{1}{a_+}, \frac{1}{\mathbb{E}[Y_1]})\\ -\lambda u & \text{ if }  u>\frac{1}{\mathbb{E}[Y_1]}.
		\end{cases}
	\end{equation*}
	Since $\psi(\theta_+)>\lambda>0$, we know that $f(u)$ is increasing for $u\in[0,\frac{1}{a_+}]$ and decreasing for $u>\frac{1}{\mathbb{E}[Y_1]}$, so we focus on its behaviour in $(\frac{1}{a_+}, \frac{1}{\mathbb{E}[Y_1]})$. Differentiating, we get for $u\in(\frac{1}{a_+},\frac{1}{\mathbb{E}[Y_1]}]$ $$f'(u)=-(\lambda + \psi^*(\frac{1}{u})-\frac{1}{u}\psi^{*'}(\frac{1}{u})).$$ 
	By Proposition \ref{prop:rappelsurlestransf}, we have  $\psi^{*'}(\frac{1}{u})=\theta_{\frac{1}{u}},$ therefore  
	\begin{equation*}
		f'(u)=-(\lambda+\theta_{\frac{1}{u}}\frac{1}{u}-\psi(\theta_{\frac{1}{u}})-\frac{1}{u}\theta_{\frac{1}{u}})=\psi(\theta_{\frac{1}{u}})-\lambda.
	\end{equation*}
	But, by Lemma \ref{prop:explosionLaplaceZT}, we know that $\theta^*$ is the unique point in $(0,\theta_+)$ such that $\psi(\theta^*) = \lambda$. So we get that $u^*$ is the only point in $(\frac{1}{a_+}, \frac{1}{\mathbb{E}[Y_1]})$ such that $f'(u^*)=0$. The monotony of $f'$ is given by Proposition \ref{prop:rappelsurlestransf} so $f$ is maximal at $u^*$ and we compute \[f(u^*)= - \theta_{\frac{1}{u^*}} + u^* (\psi(\theta_{\frac{1}{u^*}}) - \lambda) = - \theta^*. \]
	Moreover, the function $f$ is in $\mathcal{C}^2((\frac{1}{a_+},\frac{1}{\mathbb{E}[Y_1]}])$ and, for $u\in(\frac{1}{a_+},\frac{1}{\mathbb{E}[Y_1]}]$, $$f^{''}(u)=-\frac{1}{u^2}\theta^{'}_{\frac{1}{u}}\psi^{'}(\theta_{\frac{1}{u}});$$ where $\theta^{'}_{\cdot}$ denotes the derivative of the function $\theta_{\cdot}$. But $\psi^{'}(\theta_{\frac{1}{u}}) = \frac{1}{u}$ and $\theta_{\frac{1}{u}} = \psi^{*'}(\frac{1}{u}-1)$ so we have $$f^{''}(u)=\frac{\psi^{*''}(\frac{1}{u})}{u^3}.$$ 

	Now let $t\geq0$. As $$ \mathbb{E}[Y_1] < \frac{1}{u^*}< a_+,$$ we can fix $\delta>0$ such that : $$ \mathbb{E}[Y_1] < \frac{1}{u^*+\delta } \mbox{ and } \frac{1}{u^*-\delta } < \alpha_+.$$ We decompose
	$$I(t) = I_1(t)+I_2(t);$$
	where $$I_1(t)= t\int_{u^*-\delta}^{u^*+\delta}\lambda e^{-\lambda ut}\mathbb{P}(Y_{ut}>ut\frac{1}{u})\mathrm{d}u,$$ and $$I_2(t)=  t\int_{\mathbb{R}_+\backslash[u^*-\delta,u^*+\delta]}\lambda e^{-\lambda ut}\mathbb{P}(Y_{ut}>ut\frac{1}{u})\mathrm{d}u.$$
	The interval $[u^*-\delta,u^*+\delta]$ being strictly included in $(\frac{1}{a_+},\frac{1}{\mathbb{E}[Y_1]}]$, we can apply Proposition \ref{prop:continuous-time-integrolocal}  which gives 
	\begin{equation*}
		\mathbb{P}(Z_{ut}>tu\frac{1}{u}) \underset{t\rightarrow\infty}{\sim} e^{-tu\psi^*(\frac{1}{u})}\frac{1}{\theta_{\frac{1}{u}}\sqrt{2\pi\psi^{''}(\theta_{\frac{1}{u}})tu}},
	\end{equation*}
	uniformly for $u\in[u^*-\delta,u^*+\delta]$. Thus, taking $t\to\infty$ in $I_1$ we get $$I_1(t) \underset{t\to\infty}{\sim} \sqrt{t}\int_{u^*-\delta}^{u^*+\delta}\lambda e^{tf(u)}\frac{1}{\theta_{\frac{1}{u}}\sqrt{2\pi\psi^{''}(\theta_{\frac{1}{u}})u}}\mathrm{d}u.$$ 
	Given previous computations on $f$ and remarking that $u\mapsto\frac{1}{\theta_{\frac{1}{u}}\sqrt{u\psi^{''}(\theta_{\frac{1}{u}})}}$ is continuous on $(\frac{1}{a_+},\frac{1}{\mathbb{E}[Y_1]}]$, one can estimate the value of the last integral using a classical result of Laplace's integration method (cf Theorem \ref{meth_laplace} in the Appendix). It follows
	\begin{equation*}
		I_1(t)\underset{t\to\infty}{\sim}\lambda e^{-tf(u^*)}\frac{u^*}{\theta_{\frac{1}{u^*}}\sqrt{\psi^{''}(\theta_{\frac{1}{u^*}})|\psi^{*''}(\frac{1}{u^*})|}}.
	\end{equation*}
	As $\psi^{'}(\theta_a)=a$ we have for $a\in(\mathbb{E}[Y_1], a_+)$ $\psi^{''}(\theta_a)=\frac{1}{\theta^{'}_a}=\frac{1}{{\psi^*}^{''}(a)}.$ As a result
	\begin{equation}\label{equiv_i1}
		I_1(t)\underset{t\to\infty}{\sim}\frac{\lambda u^*}{\theta^*}e^{-t\theta^*}.
	\end{equation}
	Let us now show that $I_2$ is dominated by $I_1$ asymptotically. For $a,t>0$, Markov inequality gives for any $\lambda>0$ $$\mathbb{P}(Y_t>a t)\leq e^{-t(a\lambda-\psi(\lambda))}.$$ Therefore, using the definition of $\psi^*$ we have for $a,t>0$ $$\mathbb{P}(Y_t>a t)\leq e^{-t\psi^*(a)}.$$ 
	It gives the following upper bound for $I_2(t)$: $$I_2(t)\leq t\int_{\mathbb{R}_+\backslash[u^*-\delta,u^*+\delta]}\lambda e^{t[-\lambda u-u\psi^*(\frac{1}{u})]}\mathrm{d}u = t\int_{\mathbb{R}_+\backslash[u^*-\delta,u^*+\delta]}\lambda e^{tf(u)}\mathrm{d}u.$$
	We have already shown that $f$ is strictly increasing on $[0,u^*]$ and decreasing on $[u^*,+\infty[$, so there is $\eta>0$ such that $f(u)\leq f(u^*)-\eta$ when $u\notin[u^*-\delta,u^*+\delta]$. Therefore, for $t\geq1$
	$$I_2(t) \leq \lambda e^{(t-1)(f(u^*)-\eta)}\int_{0}^{+\infty} e^{f(u)}\mathrm{d}u\leq\lambda e^{(t-1)(f(u^*)-\eta)}\int_{0}^{+\infty} e^{-\lambda u}\mathrm{d}u.$$ Then by comparison $$ \underset{t\rightarrow\infty}{\lim}\frac{I_2(t)}{e^{tf(u^*)}} =0. $$ Finally, $\eqref{equiv_i1}$ gives $I_2(t)\in o(I_1(t))$, hence $I(t)\underset{t\to\infty}{\sim}I_1(t),$ completing the proof.
\end{proof}

\subsection{Overshoot}

In \cite{BertoinOvershoot99}, Bertoin, van Harn and Steutel computed the limit distributions of the overshoot and undershoot of a (non-killed) subordinator passing a given level. From these, we get the ones of a killed subordinator, conditioned on surviving. \\
For all $t>0$, we write $\tau_t$ the first passage time above level $t$ by the killed subordinator $(Z_s)_{s\geq0}$, and $V_t$ and $W_t$ the associated overshoot and undershoot, \textit{i.e}

\begin{equation*}
	\tau_t = \inf\{s>0, Z_s\geq t\};
\end{equation*}

\begin{equation*}
	V_t = \left\{
	\begin{array}{ll}
		t - Z_{\tau_{t^-}} & \mbox{if } \tau_t<T,\\
		-\infty & \mbox{otherwise;}
	\end{array}
	\right.
\end{equation*}
\begin{equation*}
	W_t = \left\{
	\begin{array}{ll}
		Z_{\tau_t} - t & \mbox{if } \tau_t<T\\
		-\infty & \mbox{otherwise.}
	\end{array}
	\right.
\end{equation*}

We first recall the result of Bertoin et al. \cite{BertoinOvershoot99} for the non-killed case. 

\begin{theorem}\label{thm:overshootBertoin}
	Let $(Y_s)_{s\geq0}$ be a (non-killed) subordinator with drift $d$, Lévy measure $\Pi$ and suppose that $0<\mu:=\mathbb{E}[Y_1] < \infty$. For a given level $t>0$, denote by $\tilde{\tau}$ the first passage time above level $t$ of $(Y_s)_{s\geq0}$, and by $\tilde{V}_t$ and $\tilde{W}_t$ the associated undershoot and overshoot. \\
	Then, as $t\to\infty$, the distribution of $(\tilde{V}_t,\tilde{W}_t)$ converges to the one of $(\tilde{V}_\infty,\tilde{W}_\infty)$ where :
	\begin{equation}\label{eqn:loiVWnonkilled}
		\mathbb{P}(\tilde{V}_\infty> v, \tilde{W}_\infty>w) = \frac{1}{\mu}\int_{v+w}^{\infty}\Pi((t,\infty))\mathrm{d}t.
	\end{equation}
\end{theorem}

To study the asymptotic behaviour as $t$ grows of the overshoots and undershoots of the killed subordinator conditioned on the survival $\{T>\tau_t\}=\{Z_{T_-}>t\}$, we introduce the Esscher transform under which such a rare event becomes a typical event. Namely we introduce the probability measure $\mathbb{Q}$ such that
\begin{equation}\label{eqn:CramertransformP}
	\mathbb{E}[f((Y_s)_{s\leq u})] = \mathbb{E}_{\mathbb{Q}}[e^{-\theta^*Y_u + \psi(\theta^*)u}f((Y_s)_{s\leq u})],
\end{equation}
for any positive and measurable function $f$ and $u>0$.  Theorem 3.9 in \cite[p.82]{KyprianouLevy} gives the following effect of the measure change on $(Y_s)_{s\geq0}$:

\begin{lem}\label{lem:YsousQ}
	Under the probability measure $\mathbb{Q}$, the process $(Y_s)_{s\geq0}$ is a subordinator with drift $d$ and Lévy measure $$\Pi^{\mathbb{Q}}(\mathrm{d}t)=e^{\theta^* t}\Pi(\mathrm{d}t).$$ Moreover, it verifies $$\mathbb{E}_{\mathbb{Q}}[Y_1]=\frac{1}{u^*}.$$
\end{lem}

We now prove the following result of convergence for the undershoot and overshoot of a killed subordinator. 

\begin{theorem}\label{thm:asympUnderOver}
	Conditionally on the survival $\{Z_{T^-}>t\}$, as $t\to\infty$, the distribution of $(V_t,W_t)$ converges to the one of $(V_\infty, W_\infty)$ given by
	\begin{equation}\label{eqn:loiVWkilled}
		\mathbb{P}_{(V_\infty,W_\infty)}(\mathrm{d}v,\mathrm{d}w) = \frac{\theta^*}{\lambda} e^{\theta^* v}  \mathbf{1}_{w>0} \Pi(\mathrm{d}w+v) dv + \mathbb{P}(V_\infty=W_\infty=0) \delta_{(0,0)}(\mathrm{d}v,\mathrm{d}w).
	\end{equation}
\end{theorem}

\begin{proof}
	Let us fix $v,w\geq0$ and let $t>0$. The goal of the proof is to compute the limit as $t\to\infty$ of the probability $$\mathbb{P}(V_t>v, W_t>w|Z_{T^-}>t) = \frac{\mathbb{P}(V_t> v, W_t>w)}{\mathbb{P}(Z_{T^-}>t)}.$$
	The asymptotic behaviour of $\mathbb{P}(Z_{T^-}>t)$ is given by $\eqref{equi_ZT}$ so we only have to compute the one of $$\mathbb{P}(V_t> v, W_t>w) = \mathbb{P}(\tilde{V}_t>v, \tilde{W}_t>w, T>\tau_t),$$ where $(\tilde{V}_t,\tilde{W}_t)$ denotes the undershoot and overshoot at $t$ of the (non-killed) subordinator $(Y_s)_{s\geq0}$. By the measure change $\eqref{eqn:CramertransformP}$, we have
	\begin{equation*}
		\begin{split}
			\mathbb{P}(\tilde{V}_t> v, \tilde{W}_t>w, T>\tau_t) &= \mathbb{E}[e^{\lambda \tau_t}\mathbf{1}_{\tilde{V}_t> v, \tilde{W}_t>w}], \\
			&= \mathbb{E}_{\mathbb{Q}}[e^{-\theta^* Y_{\tau_t}} \mathbf{1}_{\tilde{V}_t> v, \tilde{W}_t>w}].
		\end{split}
	\end{equation*}
	Recalling that $\tilde{W}_t = Y_{\tau_t} - t$, we get $$\mathbb{P}(V_t> v, W_t>w) = e^{-\theta^* t}\mathbb{E}_{\mathbb{Q}}[e^{-\theta^* \tilde{W}_t} \mathbf{1}_{\tilde{V}_t> v, \tilde{W}_t>w}].$$
	By Theorem \ref{thm:overshootBertoin}, under $\mathbb{Q}$, $(\tilde{V}_t,\tilde{W}_t)$  converges in law to $(\tilde{V}_\infty,\tilde{W}_\infty)$ as $t\to\infty$. Therefore
	\begin{equation*}
		\mathbb{P}(V_t> v, W_t>w) \underset{t\to\infty}{\sim} e^{-t\theta^*}\mathbb{E}_{\mathbb{Q}}[e^{-\theta^*\tilde{W}_\infty}\mathbf{1}_{\tilde{V}_\infty> v, \tilde{W}_\infty>w}].
	\end{equation*}
	So by Theorem \ref{thm_GD}, we obtain
	\begin{equation*}
		\begin{split}
			\lim_{t\to\infty} \mathbb{P}(V_t> v, W_t> w | Z_{T^-}>t) &= \frac{\theta^*}{u^* \lambda} \mathbb{E}_{\mathbb{Q}}[e^{-\theta^*\tilde{W}_\infty}\mathbf{1}_{\tilde{V}_\infty> v, \tilde{W}_\infty>w}], \\
			&= \mathbb{P}(V_\infty > v, W_\infty>w),
		\end{split}
	\end{equation*}
	which can be written in term of measure densities:
	\begin{equation}\label{eqn:loiVWavecQ}
		\mathbb{P}_{(V_\infty,W_\infty)}(\mathrm{d}v,\mathrm{d}w) = \frac{\theta^*}{u^*\lambda} e^{-\theta^* w} \mathbb{Q}_{(\tilde{V}_\infty,\tilde{W}_\infty)}(\mathrm{d}v,\mathrm{d}w).
	\end{equation}
	Going further, by $\eqref{eqn:loiVWnonkilled}$ and Lemma \ref{lem:YsousQ} we have 
	\begin{equation*}
		\mathbb{Q}(\tilde{V}_\infty> v, \tilde{W}_\infty>w) = u^* \int_{v+w}^{\infty}\Pi^{\mathbb{Q}}((t,\infty)) \mathrm{d}t = u^* \int_{v+w}^{\infty} \int_{t}^{\infty} e^{\theta^* s} \Pi(\mathrm{d}s) \mathrm{d}t.
	\end{equation*}
	Taking the derivative in $v$ we get 
	\begin{equation*}
		\frac{\mathrm{d}}{\mathrm{d}v}\mathbb{Q}(\tilde{V}_\infty>v, \tilde{W}_\infty>w) = u^* \int_{v+w}^{\infty} e^{\theta^* s} \Pi(\mathrm{d}s) = u^* \int_{w}^{\infty} e^{\theta^*(v+h)} \mathbf{1}_{h>0}\Pi(\mathrm{d}h+v).
	\end{equation*}
	So we can identify the law of $(\tilde{V}_\infty, \tilde{W}_\infty)$ under $\mathbb{Q}$ in term of the Lévy measure $\Pi$: 
	$$\mathbb{Q}_{(\tilde{V}_\infty,\tilde{W}_\infty)}(\mathrm{d}v,\mathrm{d}w) = u^* e^{\theta^* (v+w)} \mathbf{1}_{w>0} \Pi(\mathrm{d}w+v)\mathrm{d}v + \mathbb{Q}(\tilde{V}_\infty=\tilde{W}_\infty=0)\delta_{(0,0)}(\mathrm{d}v,\mathrm{d}w).$$
	Finally, from this last equality and formula $\eqref{eqn:loiVWavecQ}$ we deduce our result $\eqref{eqn:loiVWkilled}$ directly.
\end{proof}

As a corollary, we can compute the asymptotic law of the size of the jump of $(Z_s)_{s\geq0}$ above level $t$ conditionally on survival. For all $t>0$, we define the size of the jump of $(Z_s)_{s\geq0}$ at level $t$ by
\begin{equation*}
	J_t=\begin{cases}
		V_t + W_t & \text{ if } \tau_t<T\\ -\infty & \text{ otherwise. }
	\end{cases}
\end{equation*}

\begin{proposition}\label{prop:sautdeZasympto}
	Conditionally on $\{Z_{T^-} > t \}$, as $t\to\infty$, $J_t$ converges in law to $J_\infty$ where
	\begin{equation}\label{eqn:sautdeZasympto}
		\mathbb{P}_{J_\infty}(\mathrm{d}\ell) = \frac{1}{\lambda} (e^{\theta^* \ell}-1)\mathbf{1}_{\ell>0}\Pi(\mathrm{d}\ell) + \frac{d\theta^*}{\lambda}\delta_{0}(\mathrm{d}\ell).
	\end{equation}
\end{proposition}

Remark that this also gives the value of $\mathbb{P}(V_\infty=W_\infty=0)$ which corresponds to the limit probability for $(Z_s)_{s\geq0}$ conditioned on surviving to creep through a given level $t>0$ as $t\to\infty$:
\begin{equation}\label{overshoot_Z}
	\underset{t\rightarrow\infty}{\lim}\mathbb{P}(Z_{\tau_t}=t|Z_{T_-}>t) = \mathbb{P}(J_\infty=0) = \frac{d\theta^*}{\lambda}.
\end{equation}

\begin{proof}
	For all $t>0$ we have $J_t = V_t+W_t$ so by Theorem \ref{thm:asympUnderOver} we get the convergence of $J_t$ to $J_\infty = V_\infty + W_\infty$, and we have for all $\ell\geq0$ \[ \mathbb{P}(J_\infty > \ell ) = \mathbb{P}(V_\infty+W_\infty>\ell) = \int_{0}^{\infty} \int_{\ell-v}^{\infty} \frac{\theta^*}{\lambda} e^{\theta^*v} \mathbf{1}_{w>0}\Pi(\mathrm{d}w+v) \mathrm{d}v. \] We can rewrite the last integral with the change of variable $x=w+v$ and then use Tonelli's theorem to obtain \[ \mathbb{P}(J_\infty > \ell ) = \int_{\ell}^{\infty} \int_{0}^{x} \frac{\theta^*}{\lambda}e^{\theta^*v}\mathrm{d}v\Pi(\mathrm{d}x) = \int_{\ell}^{\infty} \frac{1}{\lambda}(e^{\theta^*x}-1)\Pi(\mathrm{d}x).\]
	To complete the proof, we still need to compute \[\mathbb{P}(J_\infty=0) = 1-\mathbb{P}(J_\infty>0)=1-\frac{1}{\lambda}\int_{0}^{\infty}(e^{\theta^*x} -1 )\Pi(\mathrm{d}x).\] But recall that $\psi(\theta^*)=\lambda$ so using $\eqref{eqn:LKforsubor}$ the Lévy-Khintchine decomposition of $\psi$, we get \[\int_{0}^{\infty}(e^{\theta^*x} -1 )\Pi(\mathrm{d}x) = \lambda - d\theta^*.\] It follows \[\mathbb{P}(J_\infty=0) = 1 - \frac{\lambda - d\theta^*}{\lambda} = \frac{d\theta^*}{\lambda}. \] \qedhere
\end{proof}

As a side note, we can compare our result with the non-killed case. In \cite{BertoinOvershoot99}, the authors also compute the asymptotic law of the jumps of $(Y_s)_{s\geq0}$ namely $\tilde{J}_\infty = \tilde{V}_\infty + \tilde{W}_\infty$ proving that \[ \mathbb{P}_{\tilde{J}_\infty}(\mathrm{d}\ell) = \frac{1}{\mathbb{E}[Y_1]} \ell \Pi(\mathrm{d}\ell). \] This imply that asymptotically, the jumps of $(Y_s)_{s\geq0}$ are linearly size-biased and by a symmetry argument that, the undershoot and overshoot $(\tilde{V}_\infty,\tilde{W}_\infty)$ are distributed as $(U\tilde{J}_\infty, (1-U)\tilde{J}_\infty)$. We proved that the behavior differs for the killed subordinator conditioned on survival. Indeed $\eqref{eqn:sautdeZasympto}$ shows that conditionally on survival, the jumps of $(Z_s)_{s\geq0}$ are also size-biased but by an exponential factor. Moreover, we can compute the laws $V_\infty$ and $W_\infty$ conditionally on $J_\infty$.

\begin{corollary}
	Conditionally on $J_\infty$, the laws of $V_\infty$ and $W_\infty = J_\infty - V_\infty$ are given by
	\begin{equation}\label{eqn:VsachantJ}
		\mathbb{P}_{V_\infty|J_\infty}(\mathrm{d}v) = \frac{\theta^*e^{\theta^*v}}{e^{\theta^*J_\infty}-1} \mathbf{1}_{v<J_\infty}\mathrm{d}v,
	\end{equation}
	\begin{equation}\label{eqn:WsachantJ}
		\mathbb{P}_{W_\infty|J_\infty}(\mathrm{d}w) = \frac{\theta^*e^{-\theta^*w}}{1-e^{-\theta^*J_\infty}}\mathbf{1}_{w<J_\infty}\mathrm{d}w. 
	\end{equation}
\end{corollary}

\begin{proof}
	By the same arguments as in the proof of Proposition \ref{prop:sautdeZasympto} we compute the joint law of $(V_\infty, J_\infty)$. For all $v,l\geq0$ we have
	\begin{equation*}
		\begin{split}
			\mathbb{P}(V_\infty>v,J_\infty>\ell)&=\mathbb{P}(V_\infty>v,V_\infty+W_\infty>\ell),\\
			&\int_{v}^{\infty}\int_{\ell-u} \frac{\theta^*}{\lambda} e^{\theta^*u}\Pi(\mathrm{d}w+u)\mathbf{1}_{w>0}\mathrm{d}u.
		\end{split}
	\end{equation*}
	Now by the change of variable $x=w+u$ and switching the integrals we get
	\begin{equation*}
		\begin{split}
			\mathbb{P}(V_\infty>v,J_\infty>\ell)&=\int_{\ell}^{\infty}\int_{v}^{\infty} \frac{\theta^*}{\lambda}e^{\theta^*u}\mathbf{1}_{x>u}\mathrm{d}u\Pi(\mathrm{d}x),\\
			&=\int_{\ell}^{\infty} \int_{v}^{\infty} \theta^*\frac{e^{\theta^*u}}{e^{\theta^*x}-1} \mathbf{1}_{x>u}\mathrm{d}u \frac{1}{\lambda}(e^{\theta^*x}-1)\Pi(\mathrm{d}x).
		\end{split}
	\end{equation*}
	So we recognize the law of $J_\infty$ $\eqref{eqn:sautdeZasympto}$ and can identify the law of $V_\infty$ conditionally on $J_\infty$, which gives $\eqref{eqn:VsachantJ}$. \\
	We can deduce the law of $W_\infty$ conditionally on $J_\infty$ by taking $w=\ell-u$ in the last computations.
\end{proof}

In particular, in comparison to the non-killed case, $\eqref{eqn:WsachantJ}$ implies that conditionally on the size of the jump $J_\infty$ the asymptotic overshoot $W_\infty$ is not uniformly distributed in $(0,J_\infty)$, but it is distributed as an truncated exponential conditioned to be smaller than $J_\infty$.

\section{Application to quasi-stationary distributions}\label{sec:appqsd}

Let $(X_t)_{t\geq0}$ be an irreducible continuous time Markov chain on $E=\mathbb{N}\cup\{\delta\}$ with $\delta$ the unique absorbing state. Let us denote by $Q=(q_{i,j})_{i,j\in\mathbb{N}\cup\{\delta\}}$ its rate matrix and by $T_\delta$ its absorption time. We suppose that $(X_t)_{t\geq0}$ verifies the hypothesis $\eqref{hyp:excursions}$ and $\eqref{eqn:cond-killedas}$, in particular this mean that \[ q_{i,\delta} = \lambda \mathbf{1}_{i=0}, \] with $\lambda>0$ the death rate from $0$ to $\delta$. We first state the following Lemma ensuring that under these conditions, the Yaglom limit $\eqref{eqn:defYaglomlim}$ does not depend on the starting position of $(X_t)_{t\geq0}$. 
\begin{lem}\label{lem:yaglomindeinitial}
	For $\lambda<\lambda_c$ the Yaglom limit $\eqref{eqn:defYaglomlim}$ does not depend on the starting point $j\in\mathbb{N}$, meaning that for any $j\in\mathbb{N}$ \[ \underset{t\to\infty}{\lim} \mathbb{P}_j(X_t=i|t<T_\delta) = \underset{t\to\infty}{\lim} \mathbb{P}_0(X_t=i|t<T_\delta).\]
\end{lem} 
This result justify the decomposition of the trajectory of $(X_t)_{t\geq0}$ under $\mathbb{P}_0$ given below, but its proof relies on technical computation which is of little interest for the understanding of our method so we give it at the end of the present section. We refer to future computations and more precisely $\eqref{eqn:valuelambdac}$ for the exact value of the critical value $\lambda_c$.

\subsection{Excursions}

The path of $(X_t)_{t\geq0}$ under law $\mathbb{P}_0$ can be decomposed as follows: when in state $0$, the process either jumps to some other state $i\in\mathbb{N}$ or jumps to $\delta$ and dies; if it survives it then makes a random excursion among the states $\mathbb{N^*}$ from which it cannot die until it reaches the state $0$ again; after a geometric number of independent excursions, it will eventually jump to $\delta$. 
In order to use known results on excursions, we work with $(\bar{X}_t)_{t\geq0}$ the non-killed version of $(X_t)_{t\geq0}$ which is an irreducible Markov chain on $\mathbb{N}$ with transition given by $\bar{Q}$ defined by \[\bar{q}_{0,0} = q_{0,0} + \lambda \mbox{ and } \bar{q}_{i,j} = q_{i,j} \mbox{ for all } i\neq j\in\mathbb{N}.\]

Let $g_1 = \inf \{t>0, \bar{X}_t\neq0\}$, for any $n\in\mathbb{N}^*$ define 
\begin{equation*}
	d_{n} := \inf\{t>g_n, \bar{X}_t = 0\}, \mbox{ and } g_{n+1} := \inf\{t>d_n, \bar{X}_t \neq 0 \}.
\end{equation*}
With those definitions, $(d_n)_{n\geq0}$ corresponds to the consecutive times of arrival at $0$ and $(g_n)_{n\geq0}$ to the times of departure from $0$, i.e. when the process $(\bar{X}_t)_{t\geq0}$ starts a new excursion above $0$. Thus, we define $(e_n(r))_{r\geq0}$ the $n$-th excursion of $(\bar{X}_t)_{t\geq0}$:
\begin{equation*}
	e_n(r) = \bar{X}_{g_n + r} \mathbf{1}_{r<L(e_n)},
\end{equation*}
where $L(e_n) = d_n - g_n$ is the length of the $n$-th excursion. By the strong Markov property, $(e_n)_{n\geq 1}$ is an i.i.d. sequence of random variables with values in the set of trajectories i.e. the set of right continuous functions with left limits from $\mathbb{R_+}$ to $\mathbb{N}$ \cite{ItoPPP}. We call their common law the \textit{excursion law} of $(\bar{X}_t)_{t\geq0}$ and we define a process $(e(r))_{r\geq0}$ following this law called a \textit{canonical excursion}. In addition, we denote by $e^L$ the canonical excursion conditioned on having a length $L(e)=L$. Finally we define $\eta$ the law of $L(e)$. In particular, the law of $e$ is characterised by its law when conditioned on having a certain length i.e. we have \[ \mathbb{E}[F(e)] = \int_{0}^{\infty} \mathbb{E}[F(e^\ell)] \eta(\dd\ell), \] for any measurable and positive function $F$.

\begin{figure}[h]
	\begin{tikzpicture}[scale=0.9]
		\draw (-1,0) -- (5.5,0);
		\draw (5.4,-0.2) -- (5.6,0.2);
		\draw (5.6,-0.2) -- (5.8,0.2); 
		\draw (5.7,0) -- (10.6,0);
		\draw (10.5,-0.2) -- (10.7,0.2);
		\draw (10.7,-0.2) -- (10.9,0.2);
		\draw[->] (10.8,0) -- (14,0);
		\draw (14,0) node[right] {$t$};
		\draw[->] (0,-1) -- (0,5);
		\draw (0,5) node[above] {$X_t$};
		\draw (0,0) node[below left] {$0$};
		\draw (0,-1) node[left] {$\delta$};
		\draw [dashed] (0,-1) -- (14,-1);
		\draw [thick] (0,0) -- (0.5,0);
		\draw [blue](0.5,0) node[below] {$g_1$};
		\draw [blue] (0.5,-0.07) -- (0.5,0.07);
		\draw [blue,dotted] (0.5,0) -- (0.5,1);
		\draw [blue,thick] (0.5,1) -- (1.2,1);
		\draw [blue,dotted] (1.2,1) -- (1.2,3);
		\draw [blue,thick] (1.2,3) -- (1.5,3);
		\draw [blue,dotted] (1.5,3) -- (1.5,0);
		\draw [thick] (1.5,0) -- (2.2,0);
		\draw [blue] (1.5,0) node[below] {$d_1$};
		\draw [blue] (1.5,-0.07) -- (1.5,0.07);
		\draw [red] (2.2,0) node[below] {$g_2$};
		\draw [red] (2.2,-0.07) -- (2.2,0.07);
		\draw [red,dotted] (2.2,0) -- (2.2,2);
		\draw [red,thick] (2.2,2) -- (3,2);
		\draw [red,dotted] (3,2) -- (3,4);
		\draw [red,thick] (3,4) -- (3.2,4);
		\draw [red,dotted] (3.2,4) -- (3.2,1);
		\draw [red,thick] (3.2,1) -- (3.5,1);
		\draw [red,dotted] (3.5,1) -- (3.5,0);
		\draw [thick] (3.5,0) -- (3.9,0);
		\draw [red] (3.5,0) node[below] {$d_2$};
		\draw [red] (3.5,-0.07) -- (3.5,0.07);
		\draw (3.9,0) node[below] {$g_3$};
		\draw (3.9,-0.07) -- (3.9,0.07);
		\draw [dotted] (3.9,0) -- (3.9,3);
		\draw (3.9,3) -- (4.5,3);
		\draw [thick] (6.5,0) -- (7.2,0);
		\draw [green,dotted] (7.2,0) -- (7.2,2);
		\draw [green] (7.2,0) node[below] {$g^{(s)}$};
		\draw [green] (7.2,-0.07) -- (7.2,0.07);
		\draw [green,thick] (7.2,2) -- (7.3,2);
		\draw [green,dotted] (7.3,2) -- (7.3,1);
		\draw [green,thick] (7.3,1) -- (7.5,1);
		\draw [green,dotted] (7.5,1) -- (7.5,4);
		\draw [green,thick] (7.5,4) -- (8,4);
		\draw [green,dotted] (8,4) -- (8,1);
		\draw [green,thick] (8,1) -- (9,1);
		\draw [green,dotted] (9,1) -- (9,2);
		\draw [green,thick] (9,2) -- (9.4,2);
		\draw [green,dotted] (9.4,2) -- (9.4,0);
		\draw [thick] (9.4,0) -- (10,0);
		\draw [green] (9.4,0) node[below] {$d^{(s)}$};
		\draw [green] (9.4,-0.07) -- (9.4,0.07);
		\draw (8,0) node[below] {$s$};
		\draw (8,-0.07) -- (8,0.07);
		\draw [thick] (11.5,0) -- (12,0);
		\draw [dotted] (12,0) -- (12,-1);
		\draw [thick] (12,-1) -- (14,-1);
		\draw (12,0) node[below right] {$T$};
		\draw (12,-0.07) -- (12,0.07);
	\end{tikzpicture}
	\caption{Sample trajectory for $(X_t)_{t\geq0}$ with its excursion decomposition.}
\end{figure}

Let us define the notion of \textit{excursion straddling time $t$}. Fix a time $t>0$, we define
\begin{equation*}
	g^{(t)} := \sup\{s<t, \bar{X}_t = 0\}, \mbox{ and } d^{(t)} := \inf\{s>t, \bar{X}_t=0\}.
\end{equation*}

If $\{\bar{X}_t>0\}$, we have $\{g^{(t)}<d^{(t)}\}$ and we define $(e^{(t)}(r))_{r\geq0}$ the excursion straddling $t$ with the identity: \[ e^{(t)}(r) = \bar{X}_{g^{(t)}+r} \mathbf{1}_{r<L(e^{(t)})}, \] where $L(e^{(t)}) = d^{(t)} - g^{(t)}$ is the length of the excursion. Else, we have $\{g^{(t)} = d^{(t)} = t\}$ and $\{\bar{X}_t=0\}$. 

We prove the following proposition stating that the law of $e^{(t)}$ the excursion straddling $t$, is the same as a canonical excursion of length $L(e^{(t)})$. 

\begin{proposition}\label{prop:excursion}
	Let $t>0$. Then, for any measurable and positive function $f$ we have 
	\begin{equation}\label{eqn:excursion}
		\mathbb{E}[f(e^{(t)},t-g^{(t)}, d^{(t)}-t)|g^{(t)}<d^{(t)}] = \mathbb{E}[f(e^{d^{(t)}-g^{(t)}},t-g^{(t)},d^{(t)}-t)|g^{(t)}<d^{(t)}].
	\end{equation}
\end{proposition}

\begin{proof}
	The excursion $e^{(t)}$ corresponds to only one of the excursions $(e_n)_{n\in\mathbb{N^*}}$, so summing over all its possible number we have 
	\begin{equation*}
		\mathbb{E}[f(e^{(t)},t-g^{(t)},d^{(t)}-t)\mathbf{1}_{\{g^{(t)}<d^{(t)}\}}] = \sum_{n=1}^{\infty}\mathbb{E}[f(e_n,t-g_n,d_n-t)\mathbf{1}_{\{g_n<t<d_n\}}].
	\end{equation*}
	Moreover, for all $n\in\mathbb{N^*}$ using the Markov property at time $g_n$:
	\begin{equation*}
		\begin{split}
			\mathbb{E}[f(e_n,t-g_n,d_n-t)\mathbf{1}_{\{g_n<t<d_n\}}] = \int_{0}^{\infty} \int_{0}^{\infty} \mathbb{E}[f(e_n,t-g,d-t)&|L(e_n)=d-g] \\ &\times\mathbb{P}(g_n\in \mathrm{d}g,d_n\in \mathrm{d}d),
		\end{split}
	\end{equation*}	
	but by the definition of the canonical excursion, the last expectation becomes 
	\begin{equation*}
		\mathbb{E}[f(e_n,t-g,d-t)|L(e_n)=d-g] = \mathbb{E}[f(e^{d-g},t-g,d-t)].
	\end{equation*}
	Finally, we can integrate back to $d^{(t)}$ and $g^{(t)}$ to get the wanted equality: 
	\begin{equation*}
		\begin{split}
			\mathbb{E}[f(e^{(t)},t-g^{(t)},d^{(t)}-t)\mathbf{1}_{\{g^{(t)}<d^{(t)}\}}] &= \sum_{n=1}^{\infty} \int_{0}^{\infty} \int_{0}^{\infty} \mathbb{E}[f(e^{d-g},t-g,d-t)] \\ &\times\mathbb{P}(g_n\in \mathrm{d}g,d_n\in \mathrm{d}d), \\
			&= \mathbb{E}[f(e^{d^{(t)}-g^{(t)}}t-g^{(t)},d^{(t)}-t)\mathbf{1}_{\{g^{(t)}<d^{(t)}\}}].
		\end{split}
	\end{equation*}
\end{proof}

\subsection{Inverse of the local time at \texorpdfstring{$0$}{0} }

We give an alternative construction of $(Z_s)_{s\geq0}$ the inverse local time of $(X_t)_{t\geq0}$ at $0$ using its excursions. To simplify future computations, the given process will not be the formal inverse of the local time but it will have the same law. 

Let $(Y_s)_{s\geq0}$ be compound Poisson process with drift $1$, whose jumps $(\xi_n)_{n\in\mathbb{N^*}}$ are distributed according to $\eta$ the law of the length $L(e)$ and occur at times described by a Poisson process $(N_s)_{s\geq0}$ of intensity $-\bar{q}_{0,0} =- (q_{0,0} + \lambda)$. Denote by $(\sigma_s)_{s\geq0}$ the inverse local time at $0$ of $(\bar{X}_t)_{t\geq0}$. The decomposition of the trajectories of $(\bar{X}_t)_{t\geq0}$ into excursions gives 
\begin{equation}\label{eqn:def Y}
	(Y_s)_{s\geq0} \egaldistr (\sigma_s)_{s\geq0}.
\end{equation} 
Moreover, remark that Lemma \ref{lem:psipourpoisson} gives us the Lévy measure $\Pi$ of $(Y_s)_{s\geq0}$: 
\begin{equation}\label{eqn:Piaveceta}
	\Pi(dx)=-\bar{q}_{0,0}\eta(dx).
\end{equation}
Now, let us introduce $T^{(\lambda)}$ an exponential random variable of parameter $\lambda$ independent of $(\bar{X}_t)_{t\geq0}$ and $(Y_s)_{s\geq0}$. We define $(Z_s)_{s\geq0}$ as the subordinator $(Y_s)_{s\geq0}$ killed at time $T^{(\lambda)}$ i.e.  
\begin{equation}\label{def:localtimeZero}
	Z_s = \left\{
	\begin{array}{ll}
		s + \sum_{j=0}^{N_s} \xi_j & \mbox{if } s<T^{(\lambda)}\\
		-\infty & \mbox{else}
	\end{array}
	\right. \mbox{ for all } s\geq0.
\end{equation}

\begin{lem}\label{lem:lienentreZetX}
	Under assumption $\mathcal{A}$, $(Z_s)_{s\geq0}$ has the same law as the inverse local time at $0$ of $(X_t)_{t\geq0}$. \\
	In particular, if we denote $\tau_t := \inf \{s>0, Z_s\geq t\}$, then
	\begin{equation}\label{eqn:lienZetX}
		\mathbb{P}_0(X_t=0|T_\delta>t)=\mathbb{P}(Z_{\tau_t^-}=t|T^{(\lambda)}>\tau_t).
	\end{equation}
\end{lem}

\begin{proof}
	Since $\mathcal{A}$ is verified, Markov property and the memorylessness property of $T^{(\lambda)}$ give for all $t\geq0$ \[ \mathbb{P}_0(T_\delta < t) = \mathbb{P}_0(L_t > T^{(\lambda)}), \] where $(L_t)_{t\geq0}$ is the local time at $0$ of $\bar{X}$. So by $\eqref{eqn:def Y}$ and $\eqref{def:localtimeZero}$ we have \[ \mathbb{P}_0(T_\delta<t) = \mathbb{P}(Y_{T^{(\lambda)}} < t) = \mathbb{P}(Z_{T^{(\lambda)_-}} < t), \] so \[ Z_{T^{(\lambda)_-}} \overset{(d)}{=} T_\delta. \]
	Furthermore, $(X_t)_{0\leq t < T_\delta} \egaldistr (\bar{X}_s)_{0\leq s \leq \sigma_{T^{(\lambda)}}}$, it automatically implies that $(\sigma_s)_{0\leq s \leq T^{(\lambda)}}$ has the same law as the inverse local time at $0$ of $(X_t)_{0\leq t < T_\delta}$ before death and therefore so does $(Z_s)_{0\leq s \leq T^{(\lambda)}}$.
\end{proof}

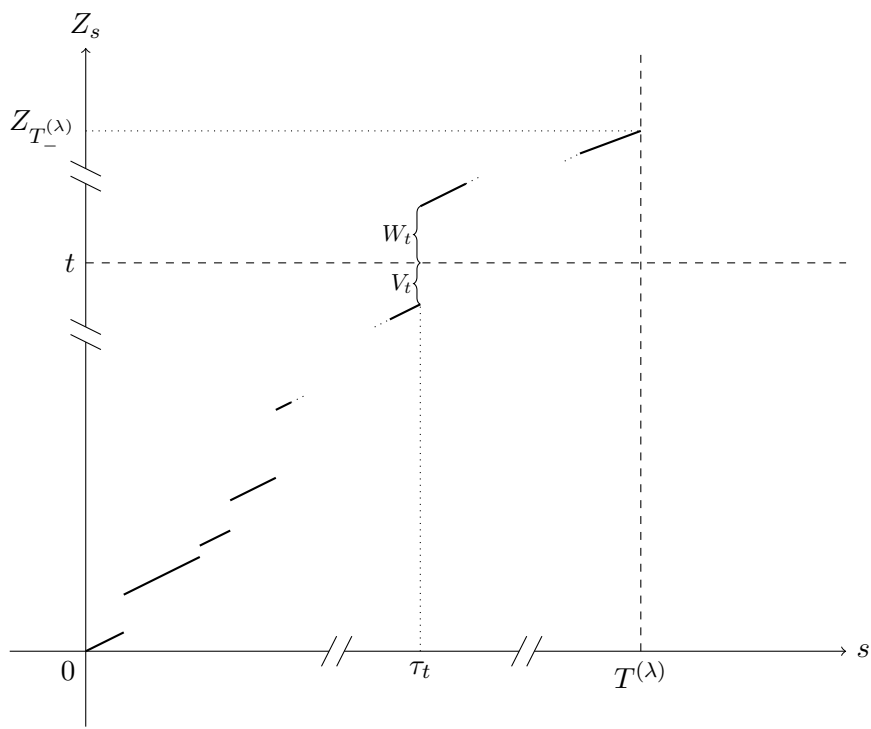
\begin{figure}[h]
	\begin{tikzpicture}[scale=1]
		\draw (0,-1) -- (0,4.1);
		\draw (0.2,4) -- (-0.2,4.2);
		\draw (0.2,4.2) -- (-0.2,4.4);
		\draw (0,4.3) -- (0,6.2);
		\draw (0.2,6.1) -- (-0.2,6.3);
		\draw (0.2,6.3) -- (-0.2,6.5);
		\draw[->] (0,6.4) -- (0,8);
		\draw (-1,0) -- (3.2,0);
		\draw (3.1,-0.2) -- (3.3,0.2);
		\draw (3.3,-0.2) -- (3.5,0.2);
		\draw (3.4,0) -- (5.7,0);
		\draw (5.6,-0.2) -- (5.8,0.2);
		\draw (5.8,-0.2) -- (6,0.2);
		\draw[->] (5.9,0) -- (10,0);
		\draw (10,0) node[right] {$s$};
		\draw (0,8) node[above] {$Z_s$};
		\draw (0,0) node[below left] {$0$};
		\draw [thick] (0,0) -- (0.5,0.25);
		\draw [thick] (0.5,0.75) -- (1.5,1.25);
		\draw [thick] (1.5,1.4) -- (1.9,1.6);
		\draw [thick] (1.9,2) -- (2.5,2.3);
		\draw [thick] (2.5,3.2) -- (2.7,3.3);
		\draw [dotted] (2.7,3.3) -- (2.9,3.4);
		\draw [thick] (4,4.4) -- (4.4,4.6);
		\draw [dotted] (3.8,4.3)-- (4,4.4);
		\draw [thick] (4.4,5.9) -- (5,6.2);
		\draw [dotted] (5,6.2) -- (5.2,6.3);
		\draw (0,5.15) node[left] {$t$};
		\draw [dashed] (0,5.15) -- (10,5.15);
		\draw [dotted] (4.4,0) -- (4.4,4.6);
		\draw (4.4,0) node[below] {$\tau_t$};
		\draw[decorate, decoration={brace,raise=0.05cm}] (4.45,4.6) -- (4.45,5.15) node[left=0.05cm,pos=0.5,scale=0.8] {$V_t$};
		\draw[decorate, decoration={brace,raise=0.05cm}] (4.45,5.15) -- (4.45,5.9) node[left=0.05cm,pos=0.5,scale=0.8] {$W_t$};
		\draw [dotted] (6.3,6.5) -- (6.5,6.6);
		\draw [thick] (6.5,6.6) -- (7.3,6.9);
		\draw [dashed] (7.3,0) -- (7.3,8);
		\draw (7.3,0) node[below] {$T^{(\lambda)}$};
		\draw [dotted] (0,6.9) -- (7.3,6.9);
		\draw (0,6.9) node[left] {$Z_{T^{(\lambda)}_-}$};
	\end{tikzpicture}
	\caption{Example of trajectory for $(Z_s)_{s\geq0}$. The jumps of $(Z_s)_{s\geq0}$ corresponds to the excursions of $(X_t)_{t\geq0}$.}
\end{figure}

Since $(Z_s)_{s\geq0}$ is a killed subordinator, the equality $ Z_{T^{(\lambda)_-}} \overset{(d)}{=} T_\delta $ implies directly that the exponential rate of survival $\gamma^*$ of the process $(X_t)_{t\geq0}$ defined in $\eqref{eqn:conditionexpokilling}$ is the point of explosion $\theta^*$ of the Laplace transform of $Z_{T^{(\lambda)_-}}$ whose behaviour is described in Lemma \ref{prop:explosionLaplaceZT}. In particular, we explicit the following characterization $\theta^*$ through the Laplace transform $\Psi_\eta$ of the distribution $\eta$ and $ \theta^{(\eta)}_+$ its point of explosion.

\begin{corollary}\label{cor:caractheta*}
	Assume $\lambda < \theta^{(\eta)}_+ - \bar{q}_{0,0} ( \Psi_\eta(\theta^{(\eta)}_+) - 1 )$, then $\theta^*$ is the unique solution of 
	\begin{equation}\label{eqn:charactheta*}
		\theta^* - \bar{q}_{0,0} (\Psi_\eta(\theta^*) - 1) = \lambda.
	\end{equation}
\end{corollary}

\begin{proof}
	We compute $\psi$ the Laplace exponent of $(Y_s)_{s\geq0}$ to apply \ref{prop:explosionLaplaceZT}. We have for all $\theta>0$ \[ \mathbb{E}[e^{\theta Y_1}] = e^{\theta} \mathbb{E}[e^{\theta\sum_{j=0}^{N_1} \xi_j}]. \] So the point of explosion of $\psi$ is $\theta^{(\eta)}_+$ and applying $\eqref{eqn:laplacecompoundpoisson}$ to the compound Poisson process $(\sum_{j=0}^{N_s} \xi_j)_{s\geq0}$ we get 
	\begin{equation}\label{eqn:psideY}
		\psi(\theta) = \theta - \bar{q}_{0,0}(\Psi_\eta(\theta) - 1),
	\end{equation}
	such that $\eqref{eqn:comportementtheta*}$ directly gives the result.
\end{proof}

Note that $\eqref{eqn:psideY}$ implies that $\theta_+:=\sup\{\theta>0, \psi(\theta)<\infty\}=\theta^{(\eta)}_+$ which then gives
\begin{equation}\label{eqn:valuelambdac}
	\lambda_c = \psi(\theta_+) = \theta^{(\eta)}_+ - \bar{q}_{0,0} ( \Psi_\eta(\theta^{(\eta)}_+) - 1 ). 
\end{equation}
We state our main theorem on the computation of the minimal QSD of $(X_t)_{t\geq0}$ through its excursions.

\begin{theorem}\label{thm:lethm}
	Assume $\gamma^*>0$. Then for all $\lambda < \lambda_c$ the process $(X_t)_{t\geq0}$ admits a minimal quasi-stationary distribution $\nu^*$ given by:
	\begin{equation}\label{eqn:qsdpour0}
		\nu^*(0) = \frac{\theta^*}{\lambda};
	\end{equation}
	and for any $i\in\mathbb{N^*}$ 
	\begin{equation}\label{eqn:qsdpouri}
		\nu^*(i) = -\bar{q}_{0,0}\frac{\theta^*}{\lambda} \int_{0}^{\infty} \int_{0}^{\infty} \mathbb{P}(e^{v+w}(v)=i) e^{\theta^* v} \eta(\mathrm{d}w+v) \mathrm{d}v.
	\end{equation}
\end{theorem}

\begin{proof}
	Recall that the minimal QSD $\nu^*$ is the asymptotic law of $(X_t)_{t\geq0}$ conditionally on survival and Lemma \ref{lem:yaglomindeinitial} ensures that we only have to compute the limit of $$\nu^*(i)=\underset{t\rightarrow\infty}{\lim}\mathbb{P}_0(X_t=i|T_\delta>t), \mbox{ for any } i\in \mathbb{N}.$$ Let us start with the case $i=0$. By Lemma \ref{lem:lienentreZetX}, for any $t>0$, \[ \mathbb{P}_0(X_t=0)=\mathbb{P}(Z_{\tau_t^-} = t |T^{(\lambda)} > \tau_t). \]
	Thus, a direct application of $\eqref{overshoot_Z}$ gives 
	\begin{equation*}
		\nu^*(0) = \underset{t\rightarrow\infty}{\lim}\mathbb{P} (Z_{\tau_t^-} = t |T^{(\lambda)} > \tau_t) = \frac{\theta^*}{\lambda}.
	\end{equation*} 

	Let $i>0$. Fix a time $t>0$. For $X_t$ to be at state $i$, the process $(X_s)_{s\geq0}$ must be at time $t$ in an excursion out of state $0$, therefore, recalling we defined $e^{(t)}$ the excursion straddling $t$, $$ \mathbb{P}_0(X_t = i| t< T_\delta) = \mathbb{P}_0(e^{(t)} (t-g^{(t)}) = i | t<T_\delta ).$$ By Proposition \ref{prop:excursion}, we have $$\mathbb{P}_0(X_t = i| t< T_\delta) = \mathbb{P} (e^{d^{(t)}-g^{(t)}} (t - g^{(t)}) = i | t < T_\delta).$$ Finally, let us link the times $g^{(t)}$ and $d^{(t)}$ to the inverse of the local time at $0$ of $(Z_s)_{s\geq0}$. By definition, we have $$d^{(t)} - g^{(t)} \overset{(d)}{=} V_t + W_t, $$ and $$ t - g^{(t)} \overset{(d)}{=} V_t, $$ where $V_t$ and $W_t$ are the undershoot and overshoot of $(Z_s)_{s\geq0}$ at level $t$. This gives us, with an application of Theorem $\ref{thm:asympUnderOver}$,
	\begin{equation*}
		\begin{split}
			\nu^*(i) &= \underset{t\rightarrow\infty}{\lim}\mathbb{P}_0(X_t=i|T_\delta>t), \\
			&=\underset{t\rightarrow\infty}{\lim}\mathbb{P} (e^{V_t + W_t} (V_t)=i | Z_{T^-} > t)=\mathbb{P} (e^{V_\infty+W_\infty} (V_\infty) = i).
 		\end{split}
	\end{equation*}
	We can integrate over the values of $(V_\infty,W_\infty)$ according to their law given by $\eqref{eqn:loiVWkilled}$ to get
	\begin{equation*}
		\begin{split}
			\nu^*(i) &= \int_{0}^{\infty} \mathbb{P}(e^{v+w}(v) = i ) \mathbb{P}_{(V_\infty,W_\infty)}(\mathrm{d}v,\mathrm{d}w), \\ 
			& = \frac{\theta^*}{\lambda} \int_{0}^{\infty} \int_{0}^{\infty} \mathbb{P}(e^{v+w}(v)=i) e^{\theta^* v} \Pi(\mathrm{d}w+v) \mathrm{d}v,
		\end{split}
	\end{equation*}
	so by $\eqref{eqn:Piaveceta}$ the last expression is equal to $\eqref{eqn:qsdpouri}$, which concludes the proof.
\end{proof}

\subsection{Proof of Lemma \ref{lem:yaglomindeinitial} }

In this section we consider an initial state $j\neq0$ and fix $\lambda<\lambda_c$. Let us show that 
\begin{equation}\label{eqn:egalYaglom}
	\lim_{t\to\infty} \mathbb{P}_j(X_t=i|t<T_\delta) = \lim_{t\to\infty} \mathbb{P}_0(X_t=i|t<T_\delta).
\end{equation}
Remark that the existence of the Yaglom limit under $\mathbb{P}_0$ is ensured by the computations done in the proof of Theorem \ref{thm:lethm}.
The proof is organized in two steps. We first compare the tails of the absorption time $T_\delta$ and of $\tau_0:=\inf\{t\geq0, X_t=0\}$. Then we use Markov Property and a dominated convergence argument to prove $\eqref{eqn:egalYaglom}$. 

\begin{lem}\label{lem:lemdelem}
	Let $j\neq0$ and denote $\alpha_0(j):=\sup\{\theta>0, \mathbb{E}_j[e^{\theta \tau_0}]<\infty]\}$. It verifies
	\begin{equation}\label{eqn:alphavstheta}
		\alpha_0(j) > \theta^*.
	\end{equation}
	In particular, we have $\lim_{t\to\infty}\frac{\mathbb{P}_j(\tau_0>t)}{\mathbb{P}_j(T_\delta>t)} =0$.
\end{lem}

\begin{proof}
	By Markov property and the definition of the excursion length of $(X_t)_{t\geq0}$ above $0$, we have \[\mathbb{P}_j(\tau_0>t) < \eta((t,+\infty)).\] Therefore, since $\lambda<\lambda_c$ we get $\alpha_0(j)\geq\theta^{(\eta)}>\theta^*$. 
\end{proof}

We can now prove $\eqref{eqn:egalYaglom}$.

\begin{proof}[Proof of Lemma \ref{lem:yaglomindeinitial}]
	Denote $\tau_0:=\inf\{t\geq0, X_t=0\}$, we have for any $i\in\mathbb{N}$
	\begin{equation}\label{eqn:decomplem} 
		\mathbb{P}_j(X_t=i|t<T_\delta) = \mathbb{P}_j(X_t=i,t\geq\tau_0|t<T_\delta) + \mathbb{P}_j(X_t=i,t\leq\tau_0|t<T_\delta). 
	\end{equation}
	On the one hand, for $t>0$ and $i\in\mathbb{N}$ we have
	\begin{equation}\label{eqn:etape}
		\mathbb{P}_j(X_t=i,t<\tau_0|t<T_\delta) = \frac{\mathbb{P}_j(X_t=i, t<\tau_0)}{\mathbb{P}_j(t<T_\delta)} \leq \frac{\mathbb{P}_j(t<\tau_0)}{\mathbb{P}_j(t<T_\delta)},
	\end{equation}
	which goes to $0$ by Lemma \ref{lem:lemdelem}. \\
	On the other hand, using the Markov property, we have
	\begin{equation*}
		\begin{split}
			\mathbb{P}_j(X_t=i, t\geq\tau_0, t<T_\delta) &= \mathbb{E}_j\big[\mathbb{E}_j[\mathbf{1}_{\{t\geq\tau_0\}} \mathbf{1}_{\{X_t=i,t<T_\delta\}} | \mathcal{F}_{\tau_0}]\big], \\
			&= \mathbb{E}_j\big[\mathbf{1}_{\{t\geq\tau_0\}} \mathbb{P}_0(X_{t-s} = i | T_\delta>t-s)\big|_{\tau_0=s}\mathbb{P}_0(T_\delta>t-s)\big|_{\tau_0=s}\big].
		\end{split}
	\end{equation*}
	It follows that \[ \mathbb{P}_j(X_t=i,t\geq\tau_0|T_\delta>t) = \mathbb{E}_j\big[\mathbf{1}_{\{t\geq\tau_0\}} \mathbb{P}_0(X_{t-s} = i | T_\delta>t-s)\big|_{\tau_0=s}\frac{\mathbb{P}_0(T_\delta>t-s)}{\mathbb{P}_j(T_\delta>t)}\big|_{\tau_0=s}\big]. \] But recall that $T_\delta \overset{(d)}{=} Z_{T^{(\lambda)_-}}$ so by Theorem \ref{thm_GD} we know that \[ \mathbb{P}_j(T_\delta>t) \underset{t\to\infty}{\sim} c e^{-\theta^* t}, \] so there exists $c_1>0$ such that $\mathbb{P}_j(T_\delta>t)\geq c_1 e^{-\theta^* t}$ for $t$ large enough. Moreover, by Chernov inequality we know that there is $c_2>0$ such that for any $0<s<t$, $\mathbb{P}_0(T_\delta>t-s) \leq c_2 e^{-\theta^*(t-s)}$. Therefore, we get that there exists $c^*>0$ such that for all $t>s$ large enough \[ \frac{\mathbb{P}_0(T_\delta>t-s)}{\mathbb{P}_j(T_\delta>t)} \leq c^* e^{\theta^* s}. \] Finally, by Lemma \ref{lem:lemdelem} we have $\mathbb{E}_j[e^{\theta^* \tau_0}] <\infty$, so we compute by dominated convergence and using the already established convergence under $\mathbb{P}_0$: 
	\begin{equation}
		\begin{split}
			\underset{t\to\infty}{\lim} \mathbb{P}_j(X_t=i,t\geq\tau_0&|T_\delta>t) \\ &= \mathbb{E}_j\big[\underset{t\to\infty}{\lim} \mathbf{1}_{\{t\geq\tau_0\}} \mathbb{P}_0(X_{t-s} = i | T_\delta>t-s)\big|_{\tau_0=s}\frac{\mathbb{P}_0(T_\delta>t-s)}{\mathbb{P}_j(T_\delta>t)}\big|_{\tau_0=s} \big], \\
			&= \mathbb{E}_j\big[\underset{t\to\infty}{\lim} \mathbb{P}_0(X_{t-s} = i | T_\delta>t-s)\big|_{\tau_0=s} \big] = \nu^*(i),
		\end{split}
	\end{equation}
	which concludes the proof.
\end{proof}

\subsection{Example of a finite cycle}

We demonstrate the use of our method with a simple example, namely a process on a finite set where the discretized trajectories are deterministic but the times of jumps are random. More precisely, we consider the process $(X_t)_{t\geq0}$ on $\{0,1,\dots,n\}\cup\{\delta\}$ defined as follows: the transition rate from $i$ to $i+1$ for any integer $0\leq i \leq n-1$ and the one from $n$ to $0$ are $1$, and when in state $0$, the process can also jumps to the death-state $\delta$ at rate $\lambda>0$. This example is simple in the sense that it contains no new results on the existence of QSD, since the existence of a unique QSD is given by the finiteness of the state space, and that the formula we get in the end can be obtained with other methods. Still it is convenient to illustrate our method since all computations on the excursions and Laplace transforms can be done explicitly, in a way that we believe easier than direct computations using the generating function.

Let us state the explicit formula for the QSD of this process.

\begin{proposition}\label{prop:qsdcycle}
	The chain $(X_t)_{t\geq0}$ admits a unique quasi-stationary distribution $\nu^*$ given by
	\begin{equation}\label{eqn:qsdcycle}
		\nu^*(i) = \frac{\theta^*}{\lambda (1-\theta^*)^n}, \text{ for } i=0,1,\dots,n,
	\end{equation}
	where the exponential rate of survival $\theta^*$ is the unique solution in $(0,\infty)$ of 
	\begin{equation}\label{eqn:thetacycle}
		\theta^* + \frac{1}{(1-\theta^*)^n} = \lambda + 1.
	\end{equation}
\end{proposition}

The proof is a direct application of Theorem \ref{thm:lethm} but it requires some computations and constructions on the excursions of the process that we give in the following lemma.

\begin{lem}\label{lem:excursionscycle}
	Let $e$ be the canonical excursion process of $(X_t)_{t\geq0}$. Then, we have the followings: 
	\begin{enumerate}
		\item The length $L(e)$ of an excursion follows a $\Gamma(n,1)$ distribution.
		\item Let $L>0$, then conditionally on $\{L(e)=L\}$ the law of the excursion $e$ is given by
		\begin{equation}\label{eqn:loiexcursioncycle}
			\mathbb{P}(e^L(t) = i) = \binom{n-1}{i-1} \left(\frac{t}{L}\right)^{i-1} \left(1-\frac{t}{L}\right)^{n-i}; \mbox{ } \forall t \in(0,L), i\in\{1,\dots,n\}.
		\end{equation}
	\end{enumerate}
\end{lem}

\begin{proof}
	The main observation we make is that, in terms of states visited, the trajectories of the excursions are deterministic: starting from state $0$ the process has to visit all the other states $1,\dots,n$ in increasing order to come back to $0$. Moreover, we know that it stays at state $i\in\{1,\dots,n\}$ for an exponential time of parameter $1$. The results follow from there.
	
	Firstly, the length of an excursion is the sum of the time that, starting from $0$, $(X_t)_{t\geq0}$ spent in each state before coming back to $0$. Therefore, we can write $$L(e)\overset{(d)}{=} \sum_{i=1}^{n} \epsilon_i,$$ where $\epsilon_1,\dots,\epsilon_n$ are i.i.d. random variables distributed as $\mathcal{E}(1)$, which proves statement $(1)$.
	
	Let us now suppose that $\{L(e)=L\}$. The jumping times of $(X_t)_{t\geq0}$ form an homogeneous Poisson process, meaning that if we know that a jump occurs in an interval $(a,b)$, its time of occurrence is uniformly distributed in this interval. Thus, we deduce that the $n$ jumping times in $(0,L]$ are distributed as an ordered vector $(\mathcal{U}_1,\dots,\mathcal{U}_{n-1}, \mathcal{U}_{n}=L)$ of $n-1$ uniform random variables on $(0,L)$. Now to get the position of the process at time $t\in(0,L)$, we only have to remark that, for any $i\in\{1,\dots,n\}$ we have $\{X_t=i\}$ if and only if the process has made exactly $i-1$ jumps since the start of its excursion. So we can write $$\mathbb{P}(e^L(t)=i) = \mathbb{P}(\mathcal{U}_{(i-1)} \leq t < \mathcal{U}_{(i)}),$$ that we rewrite $$\mathbb{P}(e^L(t)=i) = \mathbb{P}(\tilde{\mathcal{U}}_{(i-1)} \leq \frac{t}{L} < \tilde{\mathcal{U}}_{(i)}),$$ where $(\tilde{\mathcal{U}}_{(1)},\dots,\tilde{\mathcal{U}}_{(n-1)}, \tilde{\mathcal{U}}_{(n)}=1)$ is a vector of ordered uniform variables on $(0,1)$. Since the last probability is the same as the probability to get exactly $i-1$ success in $n-1$ independent Bernoulli trials of parameter $\frac{t}{L}$, this gives exactly $\eqref{eqn:loiexcursioncycle}$.
\end{proof}

We now have everything we need to compute the minimal quasi-distribution of the chain $(X_t)_{t\geq0}$. 

\begin{proof}[Proof of Proposition \ref{prop:qsdcycle}]
	Let us start with computing the exponential of survival $\theta^*$. By Corollary \ref{cor:caractheta*} it verifies \[ \theta^* - (q_{0,0} + \lambda )(\Psi_{\eta}(\theta^*)-1) = \lambda, \] where  $ q_{0,0} = -(\lambda + 1)$ is the transition rate of the process in state $0$ and $\Psi_{\eta}$ is the Laplace transform of the length $L(e)$ of the excursions that we can compute using the first point of Lemma \ref{lem:excursionscycle}: \[ \Psi_{\eta}(\theta)= \mathbb{E}[e^{\theta L(e)} ] = \frac{1}{(1-\theta)^n}.\] Therefore, $\theta^*$ is solution of \[ \theta^* - (-1) (\frac{1}{(1-\theta)^n} - 1) = \lambda, \] which is exactly $\eqref{eqn:thetacycle}$.
	We compute $\nu^*(i)$ for $i\in\{0,\dots,n\}$ with $\eqref{eqn:qsdpour0}$ and $\eqref{eqn:qsdpouri}$, it gives \[ \nu^*(0) = \frac{\theta^*}{\lambda}, \] and for $i>0$ \[\nu^*(i) = \frac{\theta^*}{\lambda} \int_{0}^{\infty} \int_{0}^{\infty} \mathbb{P}(e^{v+w}(v)=i) e^{\theta^* v} \eta(\mathrm{d}w+v) \mathrm{d}v, \] where $\eta$ is the law of $L(e)$. We proved that $L(e)$ is distributed as a $\Gamma(n,1)$, therefore $\eta$ admits the density function $f_\eta$ given for any $x\in(0,\infty)$ \[ f_\eta(x) = \frac{x^{n-1} e^{-x}}{(n-1)!}. \] In particular, we can write for any $v\in(0,\infty)$ \[ \eta(\mathrm{d}w + v) = f_\eta (w+v)\mathrm{d}w. \] Therefore, we can compute $\nu^*(i)$ using $\eqref{eqn:loiexcursioncycle}$:
	\begin{equation*}
		\begin{split}
			\nu^*(i) &= \binom{n-1}{i-1} \frac{\theta^*}{\lambda} \int_{0}^{\infty} \int_{0}^{\infty} e^{\theta^* v} (\frac{v}{v+w})^{i-1} (\frac{w}{v+w})^{n-i} 
			\frac{(v+w)^{n-1} e^{v+w}}{(n-1)!} \mathrm{d}w \mathrm{d}v, \\
			&= \frac{\theta^*}{(i-1)!(n-i)!\lambda} \int_{0}^{\infty} e^{-v(1-\theta^*)} v^{i-1}\mathrm{d}v \int_{0}^{\infty} e^{-w} w^{n-i} \mathrm{d}w,\\
			&= \frac{\theta^*}{\lambda (1-\theta^*)^i}. 
		\end{split}
	\end{equation*} 
	where we computed the two integrals by recursive integration by parts.
\end{proof}

\paragraph{Acknowledgements}
The author is indebted to its PhD advisors Bastien Mallein and Laurent Tournier for many discussions and countless answered questions. \\
We acknowledge financial support by the Investissements d'Avenir programme 10-LABX-0017, Sorbonne Paris Cité, Laboratoire d'excellence INFLAMEX.

\appendix

\section{Laplace's method of integration}

Laplace's methods to compute integrals are widely known and used. In this section we recall the classic result of approximation in Theorem \ref{meth_laplace}. Then we prove an extension of this method of approximation in a bit more general setting that corresponds to the one used in our work. To get more details and generalizations of the Laplace's method, the reader can refer to R.Wong's book \cite{WongIntegralsApproximations} on approximations of integrals (see Sections II.1, VIII.10 and IX.5).\\ 

\begin{theorem}\label{meth_laplace}
	Let $a<b\in[-\infty,+\infty]$, and $f$ and $h$ be real functions defined on $[a,b]$. Assume that $f$ is in $\mathcal{C}^2([a,b])$, and that it has a unique maximum on the segment $[a,b]$, denote $x^*$ the point where $f$ is maximal. Moreover, we assume that $h$ is continuous at $x^*$. Then, we have the following integral approximation: 
	\begin{equation}\label{laplace_classic}
		\int_{a}^{b} e^{tf(x)}h(x)\mathrm{d}x \underset{t\rightarrow+\infty}{\sim} \frac{\sqrt{2\pi}}{\sqrt{t}}e^{tf(x^*)}\frac{h(x^*)}{\sqrt{|f^{''}(x^*)|}}.
	\end{equation}
\end{theorem}

\begin{proposition}\label{Une approximation de Laplace}
	Let $f$ and $h$ be functions on $\mathbb{R}\to\mathbb{R}$. Let $A$ be a  real function on $\mathbb{R}_+$ which converges to $0$ at $+\infty$ but such that $A(t)\sqrt{t}\underset{t\rightarrow+\infty}{\longrightarrow}C\in]0,+\infty]$. Then, with same assumptions on $f$ and $h$ as in Theorem \ref{meth_laplace}, one has the following approximations: 
	\begin{equation}\label{laplace_general}
		\int_{x^*-A(t)}^{x^*+A(t)} e^{tf(x)}h(x)\mathrm{d}x \underset{t\rightarrow+\infty}{\sim} \frac{h(x^*)}{\sqrt{t}}e^{tf(x^*)}\int_{-C}^{+C}e^{-\frac{1}{2}(\sqrt{-f^{''}(x^*)}v)^2}\mathrm{d}v.
	\end{equation}
\end{proposition}

\begin{proof}
	The proof relies on approaching $f$ using Taylor's formula. Let $\delta>0$, then for $t$ large enough we have for  $x\in[x^*-A(t),x^*+A(t)]$, $$f(x^*) + \frac{1}{2}(f^{''}(x^*)-\delta)(x-x^*)^2 \leq f(x) \leq f(x^*) + \frac{1}{2}(f^{''}(x^*)+\delta)(x-x^*)^2 .$$ 
	Moreover, let $\epsilon>0$, the function $h$ being continuous at the point $x^*$, for $t$ large enough and $x\in[x^*-A(t),x^*+A(t)]$ we also have $$h(x^*)-\epsilon \leq h(x) \leq h(x^*)+\epsilon.$$ 
	Therefore, we get the following lower and upper bounds of the studied integral:
	\begin{equation*}
		\begin{split}
			(h(x^*)-\epsilon)&e^{tf(x^*)}\int_{x^*-A(t)}^{x^*+A(t)} e^{\frac{t}{2}(f^{''}(x^*)-\delta)(x-x^*)^2}\mathrm{d}x \leq \\ &\int_{x^*-A(t)}^{x^*+A(t)} e^{tf(x)}h(x)\mathrm{d}x \leq (h(x^*)+\epsilon)e^{tf(x^*)}\int_{x^*-A(t)}^{x^*+A(t)} e^{\frac{t}{2}(f^{''}(x^*)+\delta)(x-x^*)^2}\mathrm{d}x.
		\end{split}
	\end{equation*}
	Substituting $y=\sqrt{t}(x-x^*)$ in the right and left terms, we get
	\begin{equation*}
		\begin{split}
			\frac{(h(x^*)-\epsilon)e^{tf(x^*)}}{\sqrt{t}}&\int_{-A(t)\sqrt{t*}}^{A(t)\sqrt{t*}} e^{-\frac{1}{2}(y\sqrt{-f^{''}(x^*)+\delta})^2}\mathrm{d}y 
			\leq \\ \int_{x^*-A(t)}^{x^*+A(t)} &e^{tf(x)}h(x)\mathrm{d}x 
			\leq \frac{(h(x^*)+\epsilon)e^{tf(x^*)}}{\sqrt{t}}\int_{-A(t)\sqrt{t}}^{A(t)\sqrt{t}} e^{-\frac{1}{2}(y\sqrt{-f^{''}(x^*)-\delta})^2}\mathrm{d}y.
		\end{split}
	\end{equation*}
	Then, dividing by $\frac{h(x^*)}{\sqrt{t}}e^{tf(x^*)}\int_{-C}^{+C}e^{-\frac{1}{2}(\sqrt{-f^{''}(x^*)}v)^2}\mathrm{d}v$,
	\begin{equation*}
		\begin{split}
			\frac{h(x^*)-\epsilon}{h(x^*)}\frac{\int_{-A(t)\sqrt{t}}^{A(t)\sqrt{t}} e^{-\frac{1}{2}(y\sqrt{-f^{''}(x^*)+\delta})^2}\mathrm{d}y}{\int_{-C}^{+C}e^{-\frac{1}{2}(\sqrt{-f^{''}(x^*)}v)^2}\mathrm{d}v} 
			&\leq \frac{\int_{x^*-A(t)}^{x^*+A(t)} e^{tf(x)}h(x)\mathrm{d}x}{\frac{h(x^*)}{\sqrt{t}}e^{tf(x^*)}\int_{-C}^{+C}e^{-\frac{1}{2}(\sqrt{-f^{''}(x^*)}v)^2}\mathrm{d}v} \\
			&\leq \frac{h(x^*)+\epsilon}{h(x^*)}\frac{\int_{-A(t)\sqrt{t}}^{A(t)\sqrt{t}} e^{-\frac{1}{2}(y\sqrt{-f^{''}(x^*)+\delta})^2}\mathrm{d}y}{\int_{-C}^{+C}e^{-\frac{1}{2}(\sqrt{-f^{''}(x^*)}v)^2}\mathrm{d}v}.
		\end{split}
	\end{equation*}
	Now, by the definition of $A$, we have $\pm A(t)\sqrt{t} \underset{t\rightarrow\infty}{\longrightarrow} \pm C$, such that we can take $t\rightarrow\infty$ and dominated convergence theorem gives
	\begin{equation*}
		\begin{split}
			\frac{h(x^*)-\epsilon}{h(x^*)} \frac{\int_{-C}^{+C} e^{-\frac{1}{2}(y\sqrt{-f^{''}(x^*)+\delta})^2}\mathrm{d}y}{\int_{-C}^{+C}e^{-\frac{1}{2}(\sqrt{-f^{''}(x^*)}v)^2}\mathrm{d}v}
			&\leq \underset{t\rightarrow\infty}{\lim} \frac{\int_{x^*-A(t)}^{x^*+A(t)} e^{tf(x)}h(x)\mathrm{d}x}{\frac{h(x^*)}{\sqrt{t}}e^{tf(x^*)}\int_{-C}^{+C}e^{-\frac{1}{2}(\sqrt{-f^{''}(x^*)}v)^2}\mathrm{d}v} \\
			&\leq \frac{h(x^*)+\epsilon}{h(x^*)} \frac{\int_{-C}^{+C} 
			e^{-\frac{1}{2}(y\sqrt{-f^{''}(x^*)-\delta})^2}\mathrm{d}y}{\int_{-C}^{+C}e^{-\frac{1}{2}(\sqrt{-f^{''}(x^*)}v)^2}\mathrm{d}v}.
		\end{split}
	\end{equation*}
	Finally, when $\epsilon$ and $\delta$ go to $0$, one finds : 
	$$ 1 \leq \underset{t\rightarrow\infty}{\lim} \frac{\int_{x^*-A(t)}^{x^*+A(t)} e^{tf(x)}h(x)\mathrm{d}x}{\frac{\sqrt{2\pi}}{\sqrt{t}}e^{tf(x^*)}\frac{h(x^*)}{\sqrt{|f^{''}(x^*)|}}} \leq 1, $$  which gives the wanted approximation. 
\end{proof}

\begin{remarque}
	For $C=+\infty$, Proposition \ref{Une approximation de Laplace} gives $\eqref{laplace_classic}$.
\end{remarque}

\bibliographystyle{plain}
\bibliography{biblio}
\end{document}